\documentclass{article}
\usepackage[utf8]{inputenc}
\usepackage{hyperref}
\usepackage{graphicx}
\usepackage{subcaption}
\usepackage{amsmath}
\usepackage{amsfonts}
\usepackage{bm}

\usepackage{amsthm}
\newtheorem{theorem}{Theorem}
\newtheorem{corollary}{Corollary}
\newtheorem{remark}{Remark}

\usepackage[ddmmyyyy]{datetime}

\newcommand{\SSML}{\texttt{SSMLearn}}
\newcommand{\mSSM}{\texttt{fastSSM}}

\newcommand{\SSMT}{\texttt{SSMTool}}

\newcommand{\ssmorder}{m}
\newcommand{\redorder}{r}
\newcommand{\nforder}{h}

\newcommand{\fullmat}{\vct{A}}
\newcommand{\fnl}{\vct g}
\newcommand{\modf}{\vct f}
\newcommand{\delf}{\tilde{\modf}}
\newcommand{\delp}{\vct{q}}

\newcommand{\fullflow}{\vct{F}}
\newcommand{\ssmflow}{\vct{\Phi}}
\newcommand{\sampmap}{\vct{S}}

\newcommand{\fullc}{x}
\newcommand{\delc}{y}
\newcommand{\Delc}{\vct{Y}}
\newcommand{\redc}{\xi}
\newcommand{\Redc}{\vct{\Xi}}
\newcommand{\nfc}{\zeta}
\newcommand{\modc}{z}

\newcommand{\ssmmat}{\vct{M}}

\newcommand{\diagredmat}{\vct{G}}
\newcommand{\transmat}{\vct{H}}
\newcommand{\nfmat}{\vct{N}}

\newcommand{\lag}{\kappa}
\newcommand{\npoints}{N}

\newcommand{\fulldim}{n}
\newcommand{\deldim}{p}
\newcommand{\ssmdim}{{d}}
\newcommand{\obsdim}{q}

\newcommand{\modi}{k} 
\newcommand{\deli}{j}  
\newcommand{\obsi}{\ell}  
\newcommand{\geni}{i}  
\newcommand{\modiset}{K}  

\newcommand{\obs}{\mu}
\newcommand{\eigv}{\vct{e}}
\newcommand{\Eigv}{\vct{E}}
\newcommand{\tang}{\vct{T}}
\newcommand{\Van}{V}
\newcommand{\D}{\vct{D}}

\newcommand{\real}{\mathrm{Re\, }}

\newcommand{\N}{\mathbb{N}}
\newcommand{\R}{\mathbb{R}}
\newcommand{\Ce}{\mathbb{C}}
\newcommand{\e}{\mathrm{e}}

\newcommand{\mfd}{\mathcal{M}}
\newcommand{\mfdo}{\tilde{\mfd}}
\newcommand{\specsub}{\mathcal{E}}

\newcommand{\ordo}[1]{\mathcal{O}(#1)}
\newcommand{\lilordo}[1]{o(#1)}
\newcommand{\vk}{von Kármán}
\newcommand{\van}{Vandermonde matrix}
\newcommand{\vct}[1]{\bm{#1}}

\DeclareMathOperator{\diag}{diag}
\DeclareMathOperator{\range}{range}
\DeclareMathOperator{\spann}{span}

\usepackage[style=numeric,sorting=none,maxnames=10,giveninits=true]{biblatex}
\renewbibmacro{in:}{}

\title{Model reduction for nonlinearizable dynamics via delay-embedded spectral submanifolds}
\author{Joar Axås$^1$, George Haller$^1$\footnote{Corresponding author. E-mail: georgehaller@ethz.ch}}
\date{%
    $^1$Institute for Mechanical Systems, ETH Zürich, \\Leonhardstrasse 21, 8092 Zurich, Switzerland\\[2ex]%
    \today{}}

\begin{document}

\maketitle

\begin{abstract}
Delay embedding is a commonly employed technique in a wide range of data-driven model reduction methods for dynamical systems, including the Dynamic mode decomposition (DMD), the Hankel alternative view of the Koopman decomposition (HAVOK), nearest-neighbor predictions and the reduction to spectral submanifolds (SSMs).
In developing these applications, multiple authors have observed that delay embedding appears to separate the data into modes, whose orientations depend only on the spectrum of the sampled system.
In this work, we make this observation precise by proving that the eigenvectors of the delay-embedded linearized system at a fixed point are determined solely by the corresponding eigenvalues, even for multi-dimensional observables.
This implies that the tangent space of a delay-embedded invariant manifold can be predicted a priori using an estimate of the eigenvalues.
We apply our results to three datasets to identify multimodal SSMs and analyse their nonlinear modal interactions.
While SSMs are the focus of our study, these results generalize to any delay-embedded invariant manifold tangent to a set of eigenvectors at a fixed point. 
Therefore, we expect this theory to be applicable to a number of data-driven model reduction methods.
\end{abstract}

\section{Background}

Much recent effort in nonlinear dynamics has focused on data-driven model reduction methods.
Such algorithms return a simplified model of the system dynamics based on sampled trajectories from experiments or simulations.
Commonly pursued objectives for developing these methods include dimensionality reduction, sparsity, and interpretability.
Prevalent methods include the proper orthogonal decomposition (POD) \cite{lumley67,awrejcewicz04} and the dynamic mode decomposition (DMD) \cite{schmid10,kutz16,schmid22}, which fit models to data under various linearity assumptions.

Linear models cannot, however, capture characteristically nonlinear (or nonlinearizable) phenomena.
Such phenomena include the coexistence of, and the transition between, isolated and compact stationary states, such as fixed points, limit cycles, and invariant tori \cite{page19}.
To address this shortcoming, the Sparse identification of nonlinear dynamics (SINDy) algorithm fits a sparse nonlinear model to training data using a library of nonlinear functions \cite{brunton16b}.
However, the choice of this library depends on the user \cite{bertsimas22} and the coordinate system used.
Additionally, the size of the library scales up quickly with the problem dimensionality \cite{kutz22}.
While neural networks can pattern-match nonlinear phenomena \cite{chen92,daniel20,salmela21}, the models they return are often difficult to interpret and generalize poorly outside the range of training data \cite{loiseau20}.

In the last few years, spectral submanifolds (SSMs) have appeared as an alternative for model reduction in intrinsically nonlinear systems.
An SSM is the unique smoothest invariant manifold tangent to a nonresonant spectral subspace emanating from a fixed point \cite{cabre03} or a periodic or quasiperiodic orbit \cite{haro06}.
Therefore, an attracting SSM is the ideal candidate for a low-dimensional model of a nonlinear system \cite{haller16,ponsioen18}.
Concepts related to SSMs include nonlinear normal modes (NNMs) defined either as sets of periodic motions in conservative systems \cite{rosenberg62,vakakis97,kerschen09} or invariant manifolds \cite{shaw91,shaw93} and invariant spectral foliations \cite{szalai20}.
Here, we will apply SSMs, as they are unique, exist under well-defined conditions in dissipative systems, can have arbitrary dimensions, and can include internally resonant modes.

After the computation of an SSM, we can project either equations or data onto it to reduce the system to a high-fidelity model.
Automated model reduction to SSMs from equations \cite{SSMTool} can successfully predict responses to small harmonic forcing \cite{ponsioen19,ponsioen20,jain21} and bifurcations of those responses \cite{li21a,li21b}, and has also been extended to constrained mechanical systems \cite{li22}.
Recently, Ref.~\cite{cenedese21} developed a data-driven method which identifies the SSM geometry and its reduced dynamics to trajectories in an observable space \cite{SSMLearn}.
This approach also transforms the SSM-reduced dynamics to a normal form, which describes the dynamics as sparsely as possible while maintaining essential nonlinearities \cite{guckenheimer83}.
SSM-based model reduction has since been applied to both numerical and experimental datasets in fluid and structural dynamics \cite{cenedese21b,kaszas22} and control \cite{alora22}.
Ref.~\cite{axas22} showed how to improve the computational efficiency of data-driven SSM identification through a simplified formulation of the algorithm.

Delay embedding is the method of reconstructing invariant sets by viewing a select number of measurements separated by a timelag as independent observables.
This method is routinely used to aid data-driven model identification in nonlinear dynamical systems. 
Examples of model reduction methods based on delay embedding include the extended dynamic mode decomposition (DMD) \cite{williams15,dylewsky22b}, the Hankel alternative view of the Koopman decomposition (HAVOK) \cite{brunton17}, the eigensystem realization algorithm (ERA) \cite{juang85}, closure modeling \cite{pan18}, and nearest-neighbor prediction \cite{pikovsky86,sugihara90}. 
In addition, delay embedding has been extensively employed in SSM-based model reduction from data \cite{cenedese21,cenedese21b}. 
For SSMs, a closer understanding of the delay embedding map improves fits to data and produces more accurate reduced-order models \cite{axas22}.
This has motivated our present study on how invariant manifolds can be efficiently and accurately reconstructed in delay coordinates.

The main driver behind the introduction of delay embedding as a tool in dynamical systems was the discovery that it could reconstruct strange attractors from scalar measurements of chaotic systems \cite{crutchfield79,packard80}. 
Floris Takens's celebrated embedding theorem \cite{takens81} and its later extension \cite{sauer91} show that, in principle, delay embedding recovers invariant sets from the full state space in a suitable observable space under generic assumptions.
In practice, however, the choice of the timelag and embedding dimension is critical to obtain robust models \cite{broomhead86,casdagli91,yap10}.
The many methods for choosing delay parameters for chaotic attractor reconstruction include minimization of the mutual information between subsequent samples \cite{fraser86}, minimization of false nearest neighbors \cite{kennel92,abarbanel93}, and a Monte Carlo decision tree search formulation \cite{kraemer22}.

Recent work has also explored the geometric structure of delay embedded invariant sets in an effort to improve model order reduction.
For periodic data, singular value decomposition (SVD) on the delay-embedded snapshot matrix has been shown to converge to a Fourier analysis \cite{bozzo10}.
The number of delays required to recover such periodic orbits equals the number of coefficients of the Fourier spectrum \cite{pan20}.
Fitting a linear map between subsequent snapshots of such a delay-embedded periodic orbit produces a companion matrix, whose eigenvectors are given by the inverse \van{}  \cite{rowley09,drmac19}.

Furthermore, connections to convolutional coordinates \cite{kamb18} and the Frenet-Serret frame \cite{hirsh21} have been made, and an interpretation of SVD modes in delay coordinates as principal component trajectories has been proposed \cite{dylewsky22}. 
For the special case of an observed signal composed of oscillating sinusoidal functions, the observable space contains invariant spaces determined by the signal frequencies \cite{axas22}.
Recently, it was shown that subsequent components of the DMD modes of delay-embedded linear systems are related by a multiplication of the corresponding eigenvalue \cite{bronstein22}.

In this work, we explore the local dynamics close to a fixed point of a nonlinear delay-embedded system.
We show that the linear part of the delay-embedded dynamics depends solely on the corresponding eigenvalues, and not on the observable function and the full state space eigenvectors.
In particular, the eigenvectors in the observable space are given by the columns of the Vandermonde matrix of the exponential of the eigenvalues multiplied by the timelag.
Unlike available previous work, we do not attempt a linearization of the nonlinear dynamics, nor do we restrict our attention to periodic orbits.
Instead, our results imply that the nonlinear delay-embedded system has an SSM whose tangent space coincides with the column space of this Vandermonde matrix.
We exploit this structure to aid the data-driven identification of SSMs in three mechanical examples. 
We believe that these results enhance the understanding of delay embedding in reduced-order modeling and also reveal new opportunities for SSM-based model reduction. 

The structure of this paper is the following. 
First, Sect.~\ref{sec:ssm} briefly introduces SSM theory and summarizes a method for fast SSM-based data-driven modeling. 
Sect.~\ref{sec:delay} outlines a new theory for delay-embedding tangent spaces of invariant manifolds and discusses their application to SSM-based model reduction.
In Sect.~\ref{sec:applications}, we use these results to identify SSMs in examples of a 2-degree-of-freedom oscillator, simulations of multimodal vibrations in a \vk{} beam, and experiments of complex behavior in a sloshing tank.
In Sect.~\ref{sec:conclusions}, we draw conclusions from these examples and discuss possible further extensions of our theory.
Finally, Appendix~\ref{sec:proofs} contains the proofs of the results presented in Sect.~\ref{sec:delay}.

\section{Model reduction to spectral submanifolds}\label{sec:ssm}
Here, we outline previous results on rigorous model order reduction to SSMs in smooth nonlinear systems.
We also summarize \mSSM{}, the algorithm we use here to identify SSMs from data.

\subsection{Spectral submanifold theory}\label{sec:ssmtheory}
Consider a nonlinear, autonomous dynamical system of class $\mathcal{C}^{l}$, $l\in \{\N^+,\ \infty,\ a\}$, where $a$ denotes analyticity, in the form
\begin{equation}\label{eq:fullsystem}
    \dot{\vct{\fullc}} = \fullmat\vct{\fullc} + \fnl(\vct{\fullc}), \quad \vct{\fullc} \in \R^\fulldim, \quad \fnl \sim \ordo{|\vct{\fullc}|^2}, \quad \fnl : \R^\fulldim\to \R^\fulldim.
\end{equation}

Let us denote the flow map of the system by $\fullflow^t(\vct{\fullc}_0) := \vct{\fullc}(t, \vct{\fullc}_0)$, with $\vct{\fullc}(t, \vct{\fullc}_0)$ denoting the trajectory of (\ref{eq:fullsystem}) starting from $\vct{\fullc}_0$ at time $0$.
We assume that $\fullmat\in \R^{\fulldim\times \fulldim}$ is diagonalizable and that the real parts of its eigenvalues are either all strictly negative or all strictly positive. 
We take $\ssmdim$ eigenvectors of $\fullmat$ and denote their span by $\specsub{}$, i.e., a $\ssmdim$-dimensional spectral subspace of $\R^{\fulldim}$. 
In this step, we often choose the $\ssmdim$ slowest eigendirections.

Provided that the $\ssmdim$ eigenvalues corresponding to $\specsub$ are non-resonant with the remaining $\fulldim-\ssmdim$ eigenvalues of $\fullmat$, the nonlinear system has a unique smoothest, invariant manifold $\mfd$ tangent to $\specsub$ at the origin, i.e., $T_{\vct 0}\mfd = \specsub$ \cite{cabre03}.
Following \cite{haller16}, we call $\mfd$ a spectral submanifold (SSM).
In case of a resonance between $\specsub$ and the rest of the spectrum of $\fullmat$, the $\ssmdim$-dimensional SSM does not exist in general, and we must then include the resonant modal subspace in $\specsub$ to obtain a higher-dimensional SSM.
If all eigenvalues of $\fullmat$ are stable, the slowest SSM attracts nearby trajectories, which makes it suitable for model order reduction.

The open-source numerical package \SSMT{} computes SSMs from arbitrary finite-dimensional nonlinear systems \cite{SSMTool,jain21}.
More recently, the \SSML{} package was developed to find SSMs in data from nonlinear dynamical systems \cite{cenedese21,cenedese21b}.
Here, we will apply the simplified data-driven SSM algorithm \mSSM{} introduced by Ref.~\cite{axas22}.

\subsection{Fast data-driven model order reduction to spectral submanifolds}\label{sec:fastssm}
The objective of dynamics-based machine learning is to reconstruct SSMs from data, and then use SSM-reduced models for predictions of the full system response \cite{cenedese21}.
Here, we use \mSSM{} \cite{fastSSM} to identify the SSM from snaphots of trajectories in an observable space.
The procedure consists of two steps: manifold geometry detection and normal form computation.
The summary below follows Ref.~\cite{axas22}, to which we refer for further details.
Whereas that reference differentiates between the algorithm for cubic polynomial approximations of two-dimensional SSMs and its extension to arbitrary order and dimension, here, we will simply refer to both algorithms as \mSSM{}.

The SSM is parametrized in the graph style, that is, we construct $\mfd$ as a graph over the spectral subspace $\specsub$.
The data consists of snapshots $\vct{\delc}(t_{\geni}) \in \R^\deldim$ in a $\deldim$-dimensional observable space.
For each trajectory we construct the snapshot matrix $\Delc \in \R^{\deldim\times \npoints}$ from $\npoints$ snapshots as
\begin{equation}
	\Delc = \left[\begin{array}{cccc}
	\vert & \vert & & \vert \\
	\vct{\delc}(t_1) & \vct{\delc}(t_2) & \ldots & \vct{\delc}(t_{\npoints}) \\
	\vert & \vert & & \vert \\
	\end{array}\right]
\end{equation}

Let $\tang\in \R^{\deldim\times \ssmdim}$ be a matrix whose columns approximately span the SSM tangent space.
In \mSSM{}, the standard procedure is to obtain $\tang$ through SVD on the snapshot matrix $\Delc$.
However, $\tang$ can also be prescribed if the tangent space is known a priori. 
Denoting by $(\cdot)^\dagger$ the Moore-Penrose pseudoinverse, we project each snapshot $\vct{\delc}_\geni$ onto this subspace to obtain $\ssmdim$-dimensional reduced coordinates $\vct \redc$ as
\begin{equation}\label{eq:proj}
    \vct \redc = \tang^\dagger \vct{\delc}.
\end{equation}
We write $\Redc \in \Ce^{\ssmdim\times \npoints}$ for the projection of the snapshot matrix onto the tangent space. 

Next, we seek to approximate the embedding of $\mfd$ as the graph of a multivariate polynomial of order $\ssmorder$ from the data:
\begin{equation}
\begin{aligned}
    \vct{\delc}(\vct \redc) &= \ssmmat\vct \redc^{1:\ssmorder}, \quad \ssmmat &= [\ssmmat_1, \ssmmat_2, \dots, \ssmmat_{\ssmorder}], \quad \ssmmat_{\geni} \in \R^{\deldim\times \ssmdim_{\geni}},
\end{aligned}
\end{equation}
where $\ssmdim_{\geni}$ is the number of $\ssmdim$-variate monomials at order $\geni$ and the superscript in $(\cdot)^{1:l}$ denotes a vector of all monomials from order $1$ up to $l$. 
We obtain the manifold parametrization coefficients $\ssmmat\in \R^{{\deldim}\times \ssmdim_{1:\ssmorder}}$ by a polynomial regression, which yields the solution
\begin{equation}\label{eq:geometric}
	\ssmmat = \Delc(\Redc^{1:\ssmorder})^\dagger.
\end{equation}

The reduced dynamics are approximated by another $\ordo{\redorder}$ polynomial regression, with a coefficient matrix $\diagredmat \in \Ce^{\ssmdim\times \ssmdim_{1:\redorder}}$, in the form
\begin{equation}\label{eq:reddyn}
    \dot{\vct \redc} \approx \diagredmat\vct \redc^{1:\redorder}, \quad \diagredmat = \dot{\Redc}(\Redc^{1:\redorder})^\dagger.
\end{equation}

Finally, we compute the normal form \cite{guckenheimer83} of the SSM-reduced dynamics up to order $\nforder$.
This amounts to a near-identity polynomial transformation with coefficients $\transmat \in \Ce^{\ssmdim\times \ssmdim_{1:\nforder}}$ from the new coordinates $\vct \nfc \in \Ce^\ssmdim$ such that
\begin{equation}\label{eq:generalnf}
\begin{aligned}
     \vct \redc &= \transmat\vct \nfc^{1:\nforder} = \vct \nfc + \transmat_{2:\nforder}\vct \nfc^{2:\nforder}, \\
    \dot{\vct \nfc} &= \nfmat\vct \nfc^{1:\nforder} = \vct \Lambda \vct \nfc + \nfmat_{2:\nforder}\vct \nfc^{2:\nforder}.
\end{aligned}
\end{equation}
The normal form and the reduced dynamics are conjugate dynamical systems.
Therefore, we substitute (\ref{eq:generalnf}) into (\ref{eq:reddyn}) to obtain
\begin{equation}\label{eq:conjugacy}
    \D_{\vct \nfc}(\transmat\vct \nfc^{1:\nforder})\nfmat\vct \nfc^{1:\nforder} = \diagredmat(\transmat\vct \nfc^{1:\nforder})^{1:\redorder}.
\end{equation}
The matrices $\transmat$ and $\nfmat$ are computed by solving (\ref{eq:conjugacy}) recursively at increasing orders with \SSMT{} \cite{jain21}.



This procedure requires that the training data lies sufficiently close to the SSM, which can be achieved by removing initial transients from the input signal, as identified by a spectral analysis on the training data \cite{cenedese21b}.
Since the SSM built over the slowest $\ssmdim$ modes is unique and attracting, this method ensures relevant training data.

\section{Delay-embedding the tangent spaces of invariant manifolds}\label{sec:delay}
Here, we show how tangent spaces of invariant manifolds at a fixed point can be analytically recovered when the observable space arises from delay embedding of a signal. 
We also describe how the recovered tangent spaces facilitate the reconstruction of spectral submanifolds in such observable spaces.

\subsection{Theoretical results}\label{sec:delaytheory}
For the dynamical system (\ref{eq:fullsystem}), we define a scalar observable $\obs(\vct{\fullc}(t))$, where $\obs : \R^\fulldim \to \R$ is a differentiable function that returns a measured feature of system (\ref{eq:fullsystem}), such as a displacement coordinate.
In order to reconstruct features of the full phase space from the observable, we use delay embedding. 
We stack $\deldim$ consecutive measurements separated by a timelag $\tau > 0$ to create an observable space of dimension $\deldim$.
This yields a trajectory in the form $\vct{\delc}(t) = \sampmap(\vct{\fullc}(t)) \in \R^{\deldim}$, where we define the sampling map
\begin{equation}\label{eq:delay}
	\sampmap : \R^n \to \R^p,
	\quad 
    \vct{\fullc} \mapsto \left[\begin{array}{c}
    \obs(\vct{\fullc}) \\
    \obs(\fullflow^\tau(\vct{\fullc})) \\
    \obs(\fullflow^{2\tau}(\vct{\fullc})) \\
    \vdots \\
    \obs(\fullflow^{(p-1)\tau}(\vct{\fullc}))
    \end{array}\right].
\end{equation}

An important question is how invariant sets of system (\ref{eq:fullsystem}) in $\R^\fulldim$ are reproduced in the observable space $\R^\deldim$.
In particular, when the full state space trajectory $\vct{\fullc}(t)$ resides on a $\ssmdim$-dimensional invariant manifold $\mfd$, will $\vct{\delc}(t)$ also do so?
Takens's embedding theorem gives an affirmative answer.
It states that if $\obs$ is generic and no small integer multiple of $\tau$ coincides with the period of any possible periodic orbit of \eqref{eq:fullsystem} lying in $\mfd$, then for
\begin{equation}
    \deldim \ge 2d+1,
\end{equation}
the manifold $\mfd$ will have a diffeomorphic copy $\mfdo$ in $\R^{\deldim}$ via the mapping \eqref{eq:delay} \cite{takens81}. 
Whereas Takens's theorem was formulated only for scalar observable functions, this result has since been extended to multi-dimensional $\vct \obs$ as long as the total observable space dimension exceeds $2d$ \cite{deyle11}.

Both the nonlinear geometry and dynamics of $\mfd$ and the observable function influence the geometry of $\mfdo$. 
It is therefore difficult to predict its geometry for a general flow map. 
Around the fixed point $\delp = \sampmap(\vct 0)\in\R^\deldim$, however, the $\ordo{1}$ expansion of $\mfdo$, i.e., its tangent space $T_{\delp}\mfdo$, can be directly determined, as we will show next.
Note that since the flow map is the identity at the origin, $\delp$ lies on the diagonal in the observable space, with each of its identical components given by $\obs(\vct 0)$.

We start by rewriting (\ref{eq:fullsystem}) in modal coordinates:
\begin{equation}\label{eq:modsys}
    \dot{\vct \modc} = \modf(\vct \modc) = \vct \Lambda \vct \modc + \Eigv^{-1} \fnl(\Eigv \vct \modc),
\end{equation}
where $\Eigv = [\vct{\eigv}_1,\dots,\vct{\eigv}_\fulldim]$ contains the eigenvectors of $\fullmat$ and $\vct \Lambda = \diag(\lambda_1, \dots, \lambda_\fulldim)$ the corresponding eigenvalues, which we assume to be distinct.
We define modal coordinates $\vct \modc \in \Ce^\fulldim$ by letting $\vct \modc = \Eigv^{-1} \vct{\fullc}$.
Whereas the observable function is defined as a function of $\vct \fullc$, it is notationally convenient to define it as a function of $\vct \modc$, as $\obs(\vct \fullc) = \obs(\Eigv \vct\modc)$.

Let $\mfd$ be a $\ssmdim$-dimensional invariant manifold of \eqref{eq:fullsystem} intersecting the origin $\vct 0 \in \R^{\fulldim}$, where it is tangent to a set of $\ssmdim$ eigenvectors $\eigv_{1}, \eigv_{2}, \ldots, \eigv_{\ssmdim}$ of $\fullmat$ with corresponding eigenvalues $\lambda_1,\ldots,\lambda_\ssmdim$.
We define the \van{} $\vct{\Van}\in \Ce^{\deldim\times\ssmdim}$ of the $\ssmdim$ eigenvalues governing the linearized dynamics on $\mfd$ as $\Van_{\deli\modi} = \e^{\lambda_{\modi} \deli\tau}$, i.e.,
\begin{equation}\label{eq:vandermonde}
	\vct{\Van} = \left[\begin{array}{cccc} 
	1 & 1 & \dots & 1 \\
	\e^{\lambda_1 \tau} & \e^{\lambda_2 \tau} & \dots & \e^{\lambda_\ssmdim \tau} \\
	\e^{2\lambda_1 \tau} & \e^{2\lambda_2 \tau} & \dots & \e^{2\lambda_\ssmdim \tau} \\
	\vdots & \vdots & \ddots & \vdots \\
	\e^{(\deldim-1)\lambda_1 \tau} & \e^{(\deldim-1)\lambda_2 \tau} & \dots & \e^{(\deldim-1)\lambda_\ssmdim \tau}
	\end{array}\right].
\end{equation}

\begin{figure}[h]
	\centering
	\includegraphics[width=\linewidth]{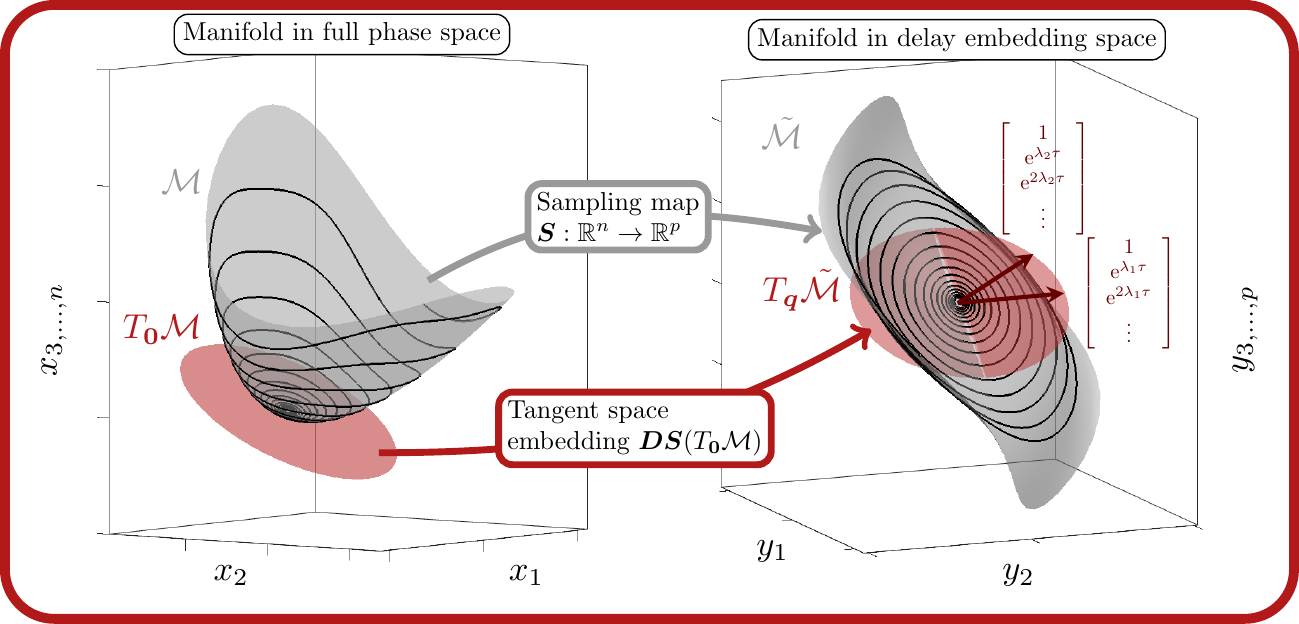}
	\caption{Delay embedding of the tangent space $T_{\vct 0}\mfd$ of an invariant manifold $\mfd$. The full state space manifold $\mfd$ (left) has a diffeomorphic copy $\mfdo$ in the observable space (right) by Takens's theorem. The shape of the reconstructed manifold $\mfdo$ depends on the flow map, but its tangent space, $T_{\delp}\mfdo$, is directly given by the eigenvalues at the fixed point, independent of the geometry of $\mfd$ and the observable function $\obs$.}\label{fig:ssmdelay}
\end{figure}

\begin{theorem}\label{thm:tangent}
Under the assumptions of a generic observable function $\obs: \R^n\to \R$ and distinct eigenvalues $\lambda_1 \ne \ldots \ne \lambda_\ssmdim$, the tangent space of the observable manifold $\mfdo$ at the fixed point can be written
\begin{equation}\label{eq:tangent}
	T_{\delp}\mfdo
	= \range{\vct{\Van}}.
\end{equation}
\end{theorem}
\begin{proof}
See Appendix \ref{sec:tangentproof}.
\end{proof}

This result is illustrated in Fig.~\ref{fig:ssmdelay}.
Note that the observable function must have full rank, as spelled out in the following remark. 
\begin{remark}\label{thm:obsrank}
For \eqref{eq:tangent} to hold, we must have $\frac{\partial \obs}{\partial \modc_{\modi}}|_{\vct 0} \ne 0$ $\forall \modi\in \{1,\ldots,\ssmdim\}$, which defines the genericity of $\obs$. 
If the gradient of the observable function is orthogonal to any of the eigenvectors $\eigv_{1}, \ldots, \eigv_{\ssmdim}$, the sampling map $\sampmap$ will not be an embedding of $\mfd$.
\end{remark}
This should be kept in mind particularly when dealing with symmetries of engineering structures, as we will show in our examples below. 

\begin{theorem}\label{thm:diag}
The columns of $\vct{\Van}$ are eigenvectors of the linearized delay-embedded system at the fixed point. 
Indeed, the dynamics in the observable space can be written
\begin{equation}
	\dot{\vct{\delc}} = \vct{\Van} \vct\Lambda \vct{\Van}^{\dagger}(\vct{\delc}-\delp) + \lilordo{|\vct{\delc}-\delp|}
\end{equation}
\end{theorem}
\begin{proof}
See Appendix \ref{sec:diagproof}.
\end{proof}
\begin{corollary}\label{thm:noobs}
In the observable space $\R^\deldim$, the timelag $\tau$ and the eigenvalues $\lambda_{\modi}$ fully determine the tangent space and the linear part of the dynamics. 
In particular, the linear dynamics are independent of both the full eigenvectors and the observable function. 
\end{corollary}

In the following, we will demonstrate how this structure can be exploited for parametrizing spectral submanifolds from data, when the corresponding eigenvalues are approximately known.

Finally, when the observable function is multi-dimensional, the tangent space is influenced by the relative dependency of each component $\obs_{\obsi}$ of the observable function on each modal coordinate $\modc_{\modi}$.

\begin{theorem}\label{thm:multidim}
For a multidimensional observable $\vct \obs : \R^\fulldim \to \R^\obsdim$ with components $\obs_1,\ldots,\obs_\obsdim$, the tangent space $T_{\delp}\mfdo \subset \R^{\deldim\obsdim}$ can be expressed as
\begin{equation}\label{eq:multidimtangent}
	T_{\delp}\mfdo = 
	\range
	\left[\begin{array}{c}
	\vct{\Van}
	\diag\left(\left.\frac{\partial \obs_1}{\partial \vct \modc}\right|_{\vct 0}\right) \\ 
	\vdots \\ 
	\vct{\Van}
	\diag\left(\left.\frac{\partial \obs_{\obsdim}}{\partial \vct \modc}\right|_{\vct 0}\right)
	\end{array}\right].
\end{equation}
\end{theorem}
\begin{proof}
See Appendix \ref{sec:multivarproof}.
\end{proof}

When the observable function is a set of displacements, the linearized multi-dimensional observable function $\left.\frac{\partial \vct\obs}{\partial \vct \modc}\right|_{\vct 0}$ corresponds to the mode shapes of the system in terms of those displacements.
Therefore, if the mode shapes and eigenvalues of the observed system are known, we can directly compute the tangent space of $\mfdo$.
In the special case of a scalar observable, the tangent space is independent of the observable function and we do not need any information about the mode shapes.


\subsection{Delay-embedded spectral submanifold reconstruction}\label{sec:picktimelag}
These theoretical results can be exploited as a constraint to aid SSM identification from data.
In the case of a scalar signal and with the eigenvalues of interest $\lambda_1,\ldots,\lambda_\ssmdim$ approximately known, we select the matrix representation of the tangent space $\tang$ appearing in \eqref{eq:proj} as the \van{} (\ref{eq:vandermonde}), i.e.,
\begin{equation}
	\tang := \vct\Van.
\end{equation}

\begin{figure}[h]
	\centering
	\includegraphics[width=\linewidth]{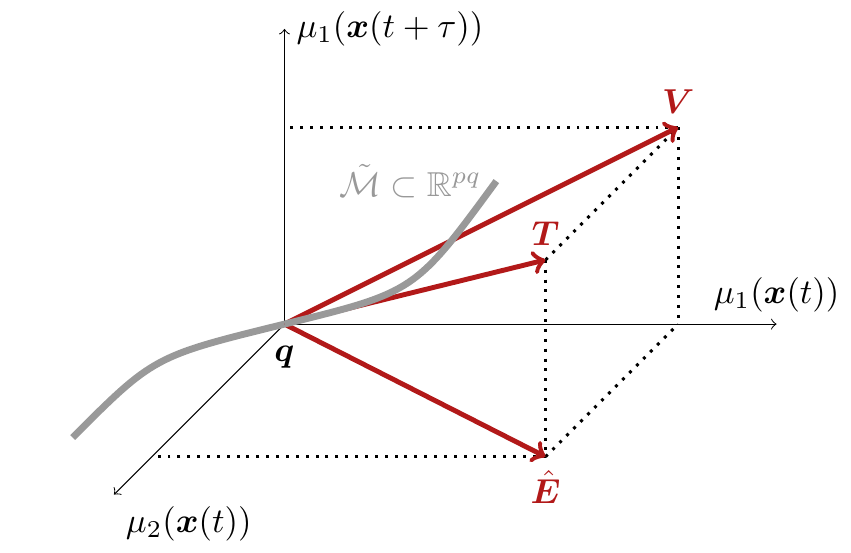}
	\caption{The tangent space $T_{\delp}\mfdo$ of the delay-embedded manifold $\mfdo$ for a $\obsdim$-dimensional observable function $\vct\obs$ is the range of the matrix $\tang$, defined as the columnwise Kronecker product of the \van{} $\vct\Van$ and the mode shapes $\hat{\Eigv}$ in terms of the observable.}\label{fig:multidim}
\end{figure}

We have seen that the gradient of a multi-dimensional observable function $\vct\obs$ enters the expression for the tangent space \eqref{eq:multidimtangent}. 
While an expression for this gradient is typically not available in experiments, mode shapes $\hat{\Eigv}=[\hat{\eigv}_1,\ldots,\hat{\eigv}_\ssmdim] \in \R^{\obsdim \times \ssmdim}$ are often known from theory or obtained experimentally.
Here, each mode shape $\hat{\eigv}_\modi \in \R^\obsdim$ describes how the eigenvector $\eigv_\modi \in \R^\fulldim$ is observed.
Specifically, they are related by
\begin{equation}\label{eq:c}
	\hat{\eigv}_\modi = c_\modi\left.\frac{\partial \vct\obs}{\partial \vct\fullc}\right|_{\vct 0}\eigv_\modi,
\end{equation}
where $c_\modi \in \Ce$ is a nonzero constant that only rescales the eigenvectors.
We select $\tang$ as the columnwise Kronecker product of the \van{} and the observable mode shapes, i.e., 
\begin{equation}\label{eq:multidimeig}
	\tang := \left[\begin{array}{c}
	\vct{\Van}
	\diag\left(\hat{\eigv}_1\right) \\ 
	\vdots \\ 
	\vct{\Van}
	\diag\left(\hat{\eigv}_\ssmdim\right)
	\end{array}\right].
\end{equation}
A sketch of the geometry is shown in Fig.~\ref{fig:multidim}.
In the case of unknown mode shapes, it may be possible to first project low-amplitude data onto the delay-embedded eigenvectors and then extract the observable mode shapes via SVD, although we do not explore this idea further in this work.

Prescribing $\tang$ and projecting the data onto its columns yields modal reduced coordinates.
This diagonalization of the system simplifies the learning of the geometry and the reduced dynamics of the SSM via the algorithm outlined in Sect. \ref{sec:fastssm}. 

Choosing proper delay-embedding parameters to reconstruct nonlinear systems can be a challenge. 
For the linear part of the system, however, our results suggest picking the timelag $\tau$ and embedding dimensionality $\deldim$ so as to obtain numerically favorable reduced coordinates along the SSM. 
We ideally want the columns of the \van{} \eqref{eq:vandermonde} to be orthogonal in order to maximize the signal-to-noise ratio in each of the observed modes.
To this end, we formulate a minimization problem,
\begin{equation}\label{eq:delaycost}
	(\lag^\star,\deldim^\star) = \mathop{\mathrm{argmin}}_{\lag,\deldim\in\N^+} \left\|\vct{\Van}(\lag\Delta t,\deldim)^\top\vct{\Van}(\lag\Delta t,\deldim) - \vct I\right\|_\mathrm{F},
\end{equation}
in which the columns of $\vct{\Van}$ are normalized and $\|\cdot\|_\mathrm{F}$ denotes the Frobenius norm.
Since the timelag is an integer multiple of the sampling timestep, $\tau=\lag\Delta t$, \eqref{eq:delaycost} defines an optimization over a set of discrete variables which can be solved simply by brute force. 

Bearing in mind the nonlinear part of the system, however, an optimal choice of delay parameters is not as straightforward. 
Increasing the timelag and embedding dimension tends to curve the SSM, requiring higher orders of approximation and in extreme cases folding the manifold, so that it can no longer be parametrized as a graph.
Taking into account the nonlinear part of the SSM reconstruction, therefore, we typically want the total delay embedding to be as small as possible.
While solving \eqref{eq:delaycost} gives some guidance, a suitable choice of $\tau$ and $\deldim$ will also depend on the nonlinearity of the system in the data range and the amount of signal noise.

\section{Applications}\label{sec:applications}
We now apply our method to three datasets: two from simulations and one from experiments. 
The eigenvalues in these examples are known either from theory or simulations.
We infer the delay-embedded tangent space accordingly before parametrizing the SSM.
The examples include an oscillator chain, a clamped-clamped beam and tank sloshing.

\subsection{Two-degree-of-freedom oscillator with nonlinear springs}
As our first example, we consider an oscillator chain of two masses, both attached with linear springs to each other and to the ground.
In addition, the spring connecting the left mass to the ground has a quadratic softening nonlinearity and the spring connecting the masses is of cubic hardening type. 
The masses and linear spring stiffnesses are set to 1, the softening parameter is $-2$ and the hardening parameter is $1$.
Each of the springs also has a linear damping coefficient of $0.03$.
Fig.~\ref{fig:oscsetup} shows the configuration.
The sampling time is $\Delta t=0.1$~s.
\begin{figure}[h]
	\centering
	\subfloat[\label{fig:oscsetup}]{\includegraphics[width=0.33\linewidth]{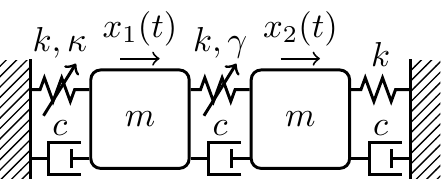}}
	\subfloat[\label{fig:oscssmfull}]{\includegraphics[width=0.33\linewidth]{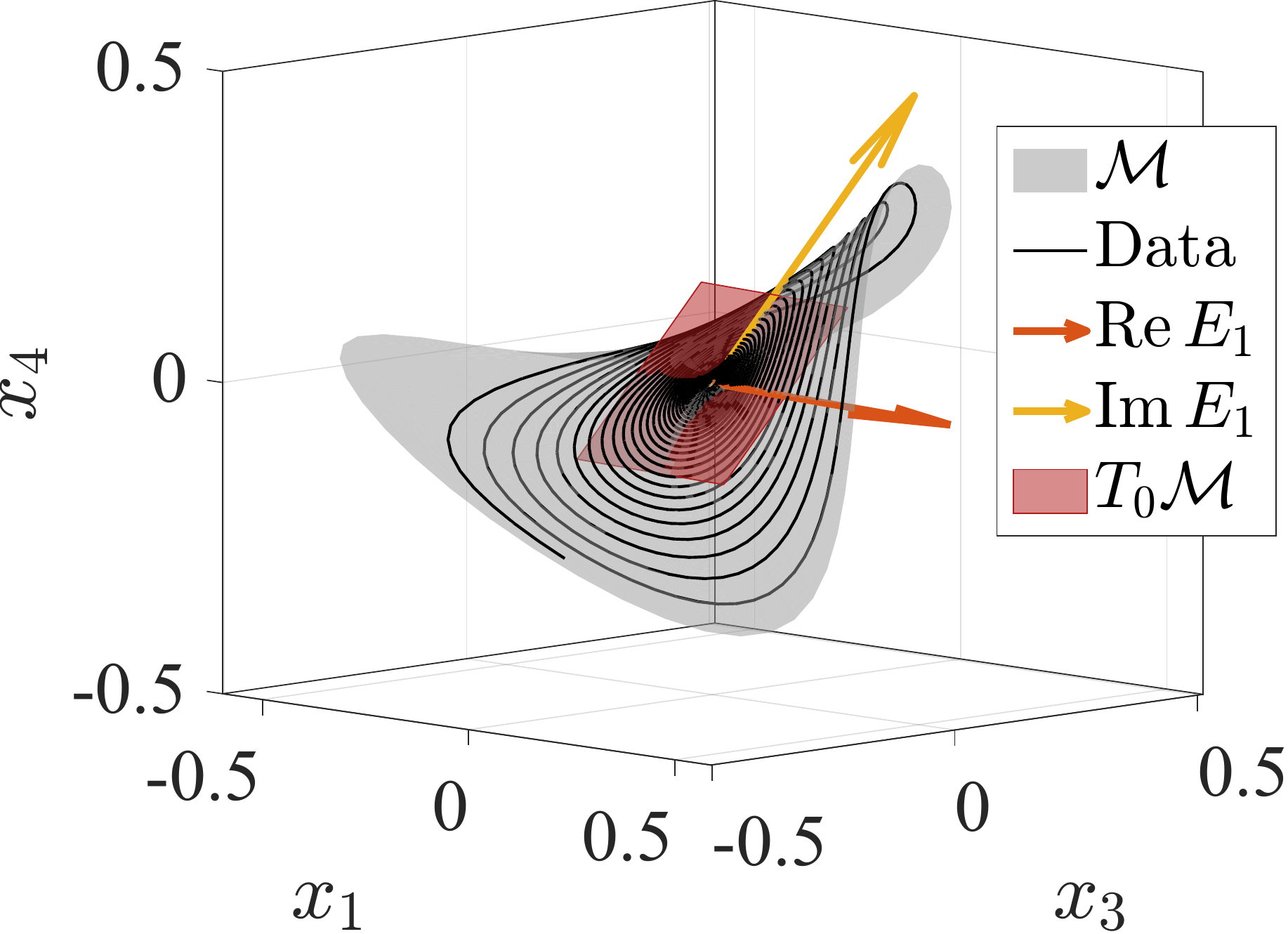}}
	\subfloat[\label{fig:oscssmdelay15}]{\includegraphics[width=0.33\linewidth]{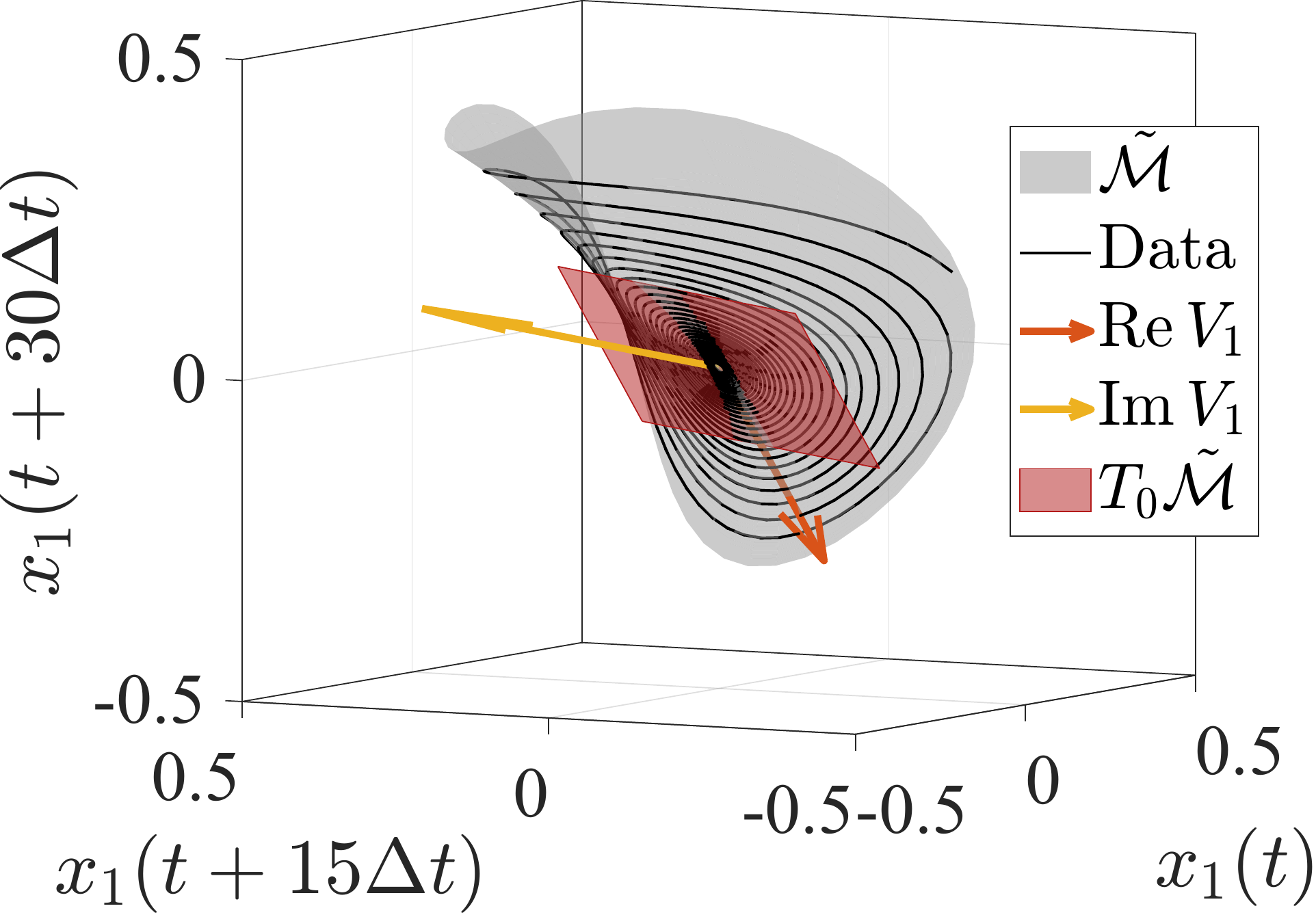}}
	\caption{(a) Setup for the two-degree-of-freedom oscillator example with two nonlinear springs. (b) The slow 2D SSM (gray) in the full state space, along with its tangent space (red). (c) The delay-embedded SSM in the observable space.}
\end{figure}

We compute an initial condition on the slow 2D SSM for the single training trajectory using \SSMT{}. 
Our observable function is the first mass displacement, $\obs(\vct{\fullc}) = \fullc_1$.
The trajectory in the full phase space and the SSM are shown in Fig.~\ref{fig:oscssmfull}.
The first two eigenvectors span the tangent space of the SSM.

Next, we delay embed the trajectory with a timelag $\tau=15\Delta t$ and embedding dimension $\deldim=5$, and seek the 2D SSM in this observable space using \mSSM{}.
We obtain reduced coordinates by projection of the trajectory data onto the columns of the \van{} $\vct{\Van}$ as predicted by our theory.
Fig.~\ref{fig:oscssmdelay15} shows the SSM in the first three coordinates of the observable space. 
Indeed, the tangent space of this observable space is identical to the column space of $\vct{\Van}$.
\begin{figure}[h]
	\centering
	\subfloat[\label{fig:oscssmdelayx2}]{\includegraphics[width=0.33\linewidth]{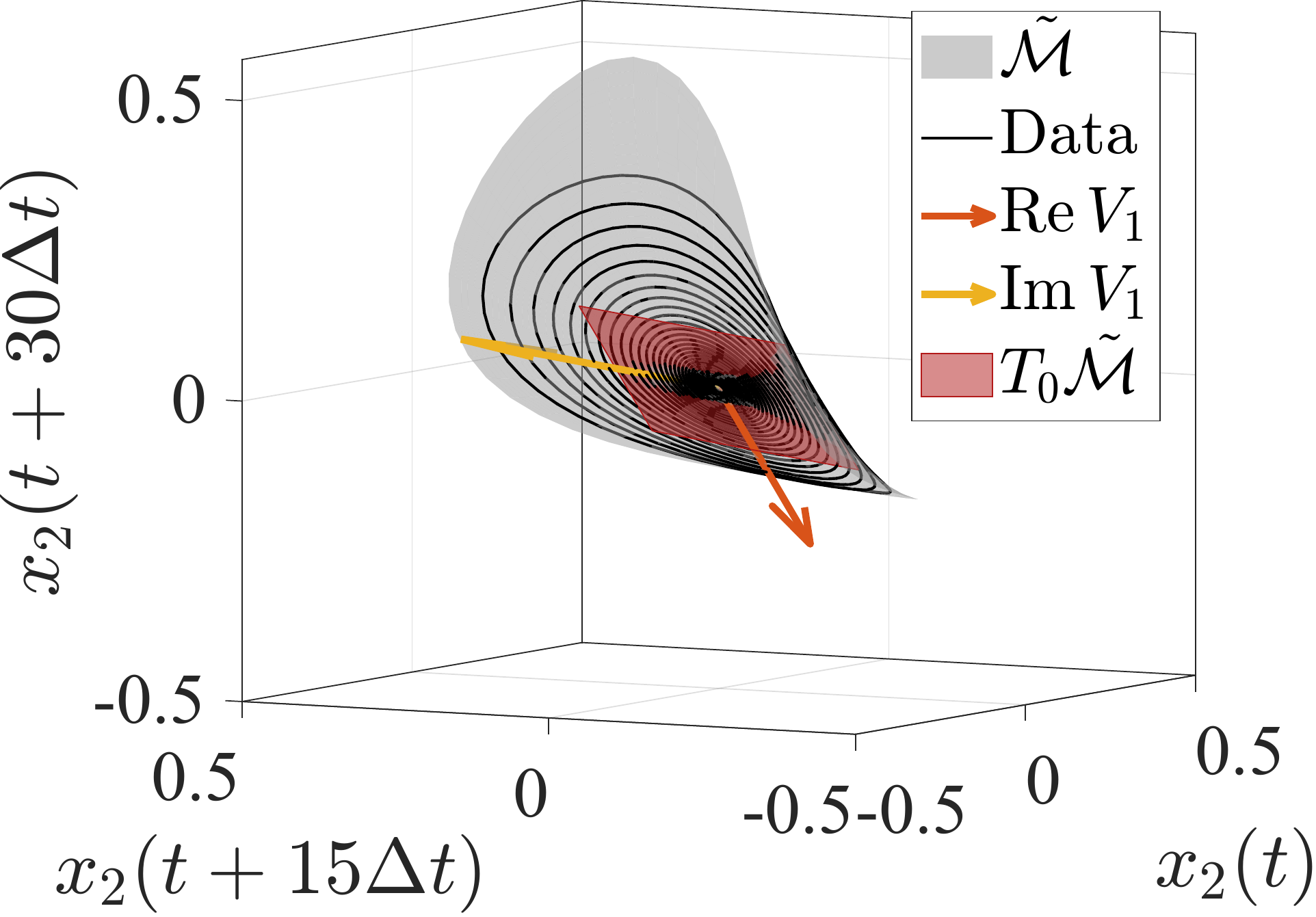}}
	\subfloat[\label{fig:oscssmdelayx3plusx4}]{\includegraphics[width=0.33\linewidth]{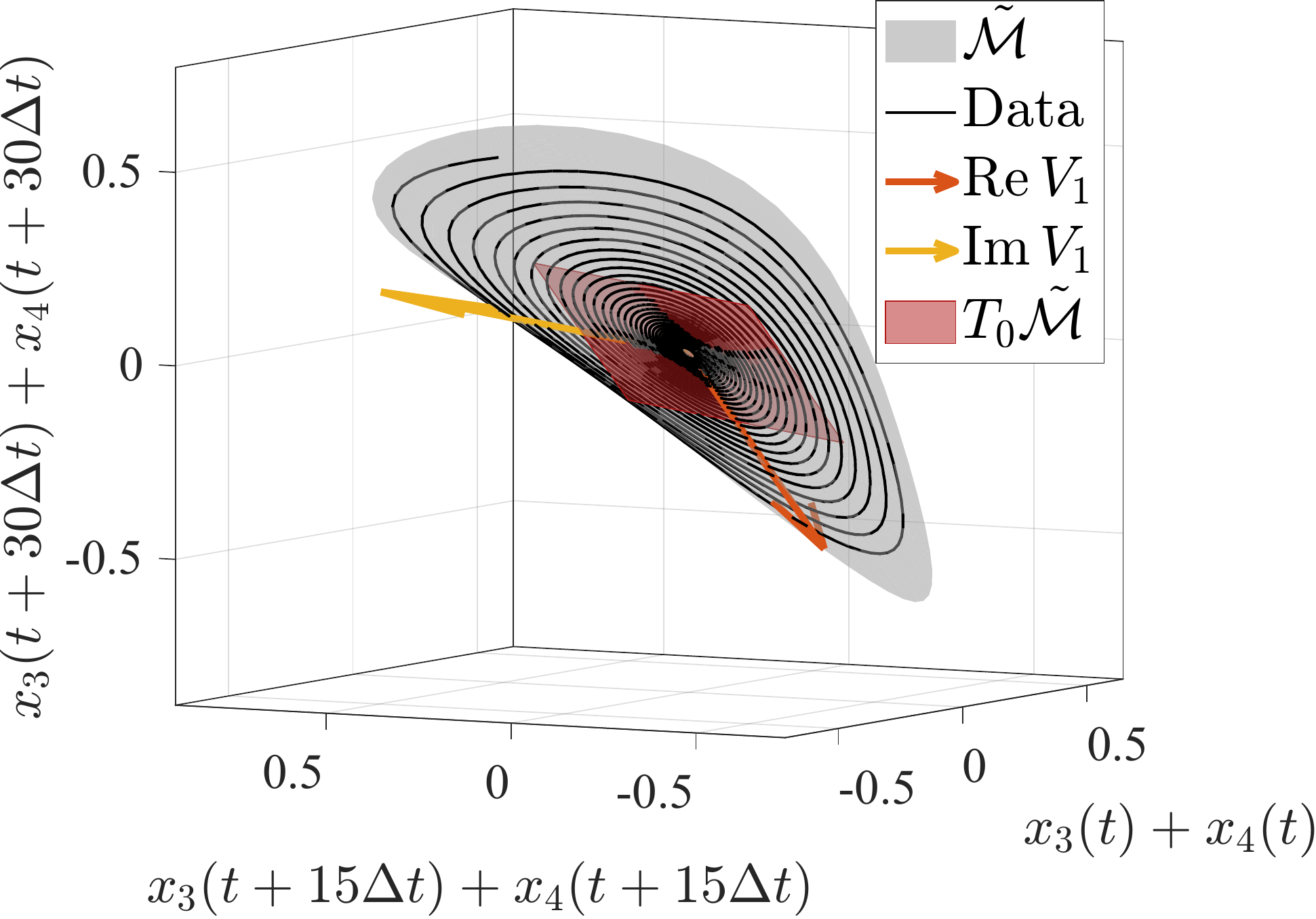}}
	\subfloat[\label{fig:oscssmdelayx2-x1}]{\includegraphics[width=0.33\linewidth]{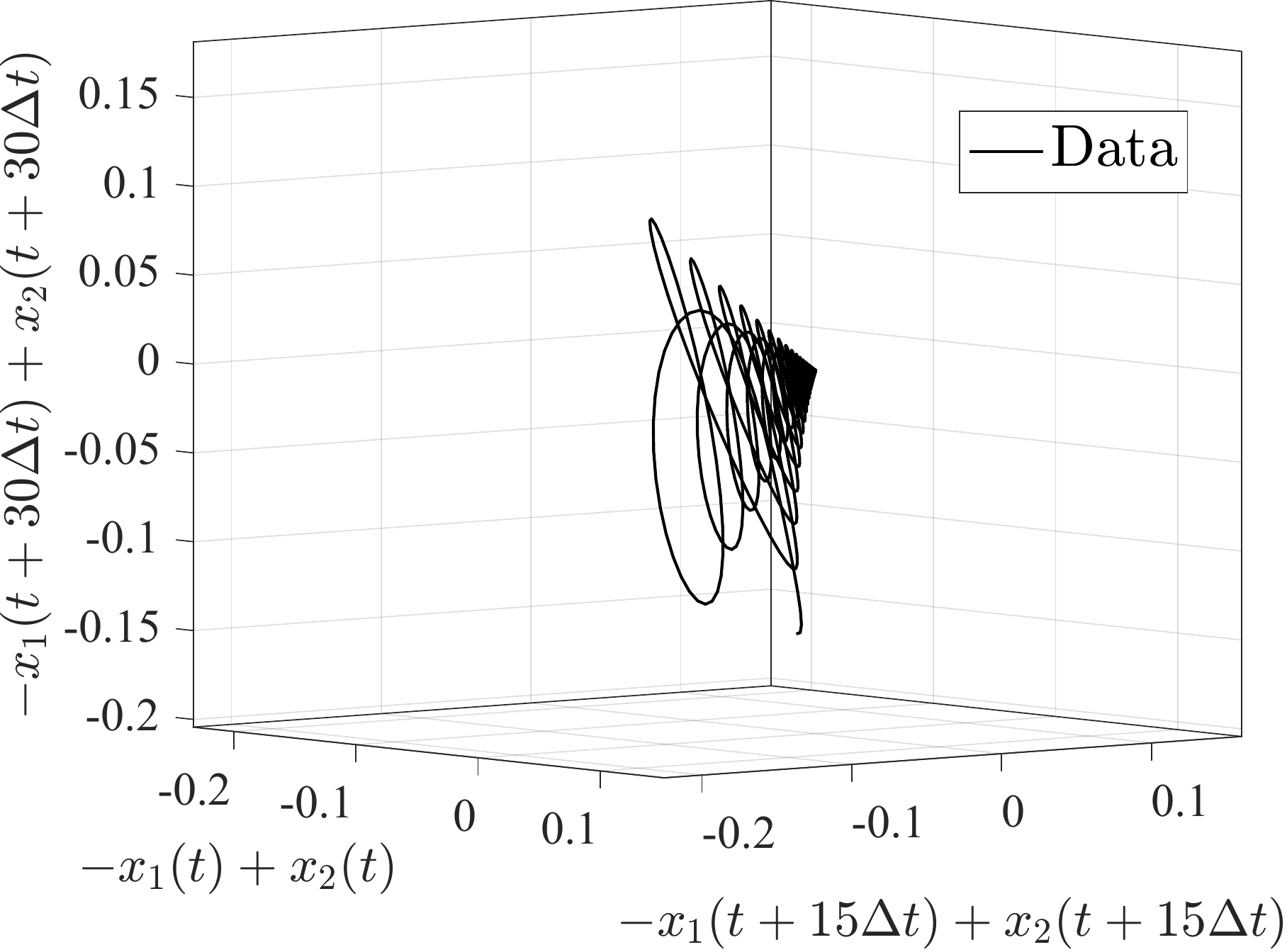}}
	\caption{(a,b) Changing the observable function leads to different SSM geometries, but the tangent space remains the same as in Fig.~\ref{fig:oscssmdelay15}. (c) A nongeneric observable function however, observing only the second mode, does not embed the manifold.}
\end{figure}

Corollary \ref{thm:noobs} predicts that this tangent space will be independent of the observable function, provided that it is generic.
To illustrate this, we plot the delay-embedded SSM for various observable functions, $\obs(\vct{\fullc}) = \fullc_1$ (Fig.~\ref{fig:oscssmdelay15}), $\obs(\vct{\fullc}) = \fullc_2$ (Fig.~\ref{fig:oscssmdelayx2}), and $\obs(\vct{\fullc}) = \dot \fullc_1+\dot \fullc_2$ (Fig.~\ref{fig:oscssmdelayx3plusx4}).
These different observable functions clearly produce different SSM geometries, but the eigenvectors and tangent spaces of the manifolds all agree. 

One exception is when we observe the distance between the masses, $\obs(\vct{\fullc}) = \fullc_2-\fullc_1$ (Fig.~\ref{fig:oscssmdelayx2-x1}).
In this case, the delay-embedded trajectory no longer lies on an invariant manifold, as is evident by the nonsmooth cusp in the data at the origin.
The reason is that this observable is non-generic precisely in the sense of our theory; it coincides with the mode shape of the second, fast mode of the full system.
This means that the observable function acts orthogonally to the slow SSM at the fixed point and thus the delay mapping is not an embedding, by Remark \ref{thm:obsrank}.

Next, we pick $\obs(\vct{\fullc}) = \fullc_2$ and use \mSSM{} to approximate the cubic order reduced dynamics on the SSM from the data. 
Computing the normal form yields 
\begin{equation}
	\left[\begin{array}{c} \dot{\rho}_{1} \\ \dot{\theta}_{1} \\  \end{array}\right]=\left[\begin{array}{c} -0.0014\,{\rho _{1}}^3-0.0148\,\rho _{1}\\ 1.0025-0.0919\,{\rho _{1}}^2 \end{array}\right].
\end{equation}
The trajectory projected onto the columns of $\vct{\Van}$ is shown in Fig.~\ref{fig:oscmodal}.
Integrating the obtained normal form and mapping back to the observable space yields a good reconstruction of the training data.
\begin{figure}[h]
	\centering
	\subfloat[\label{fig:oscmodal}]{\includegraphics[width=0.33\linewidth]{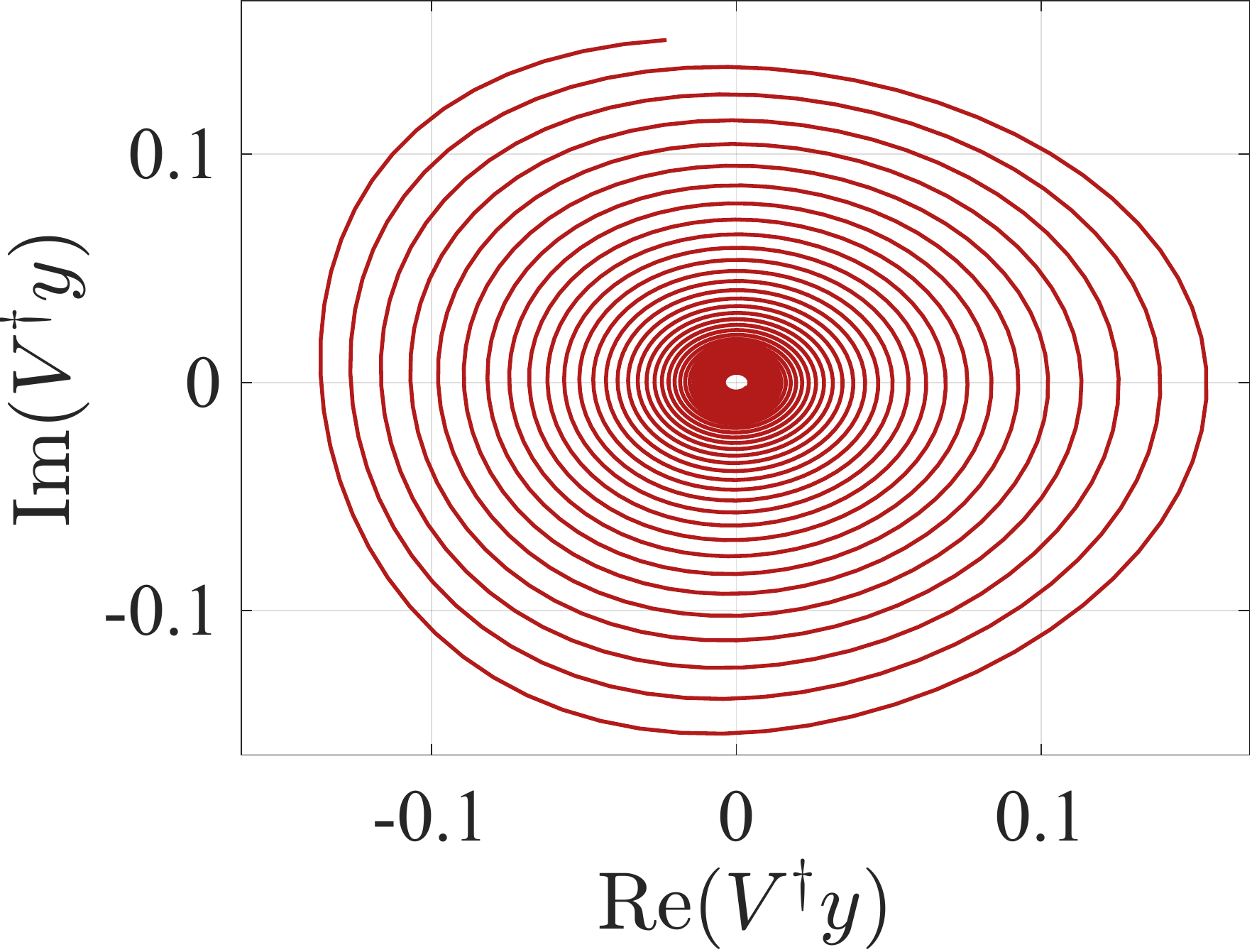}}
	\subfloat[\label{fig:oscdecay}]{\includegraphics[width=0.33\linewidth]{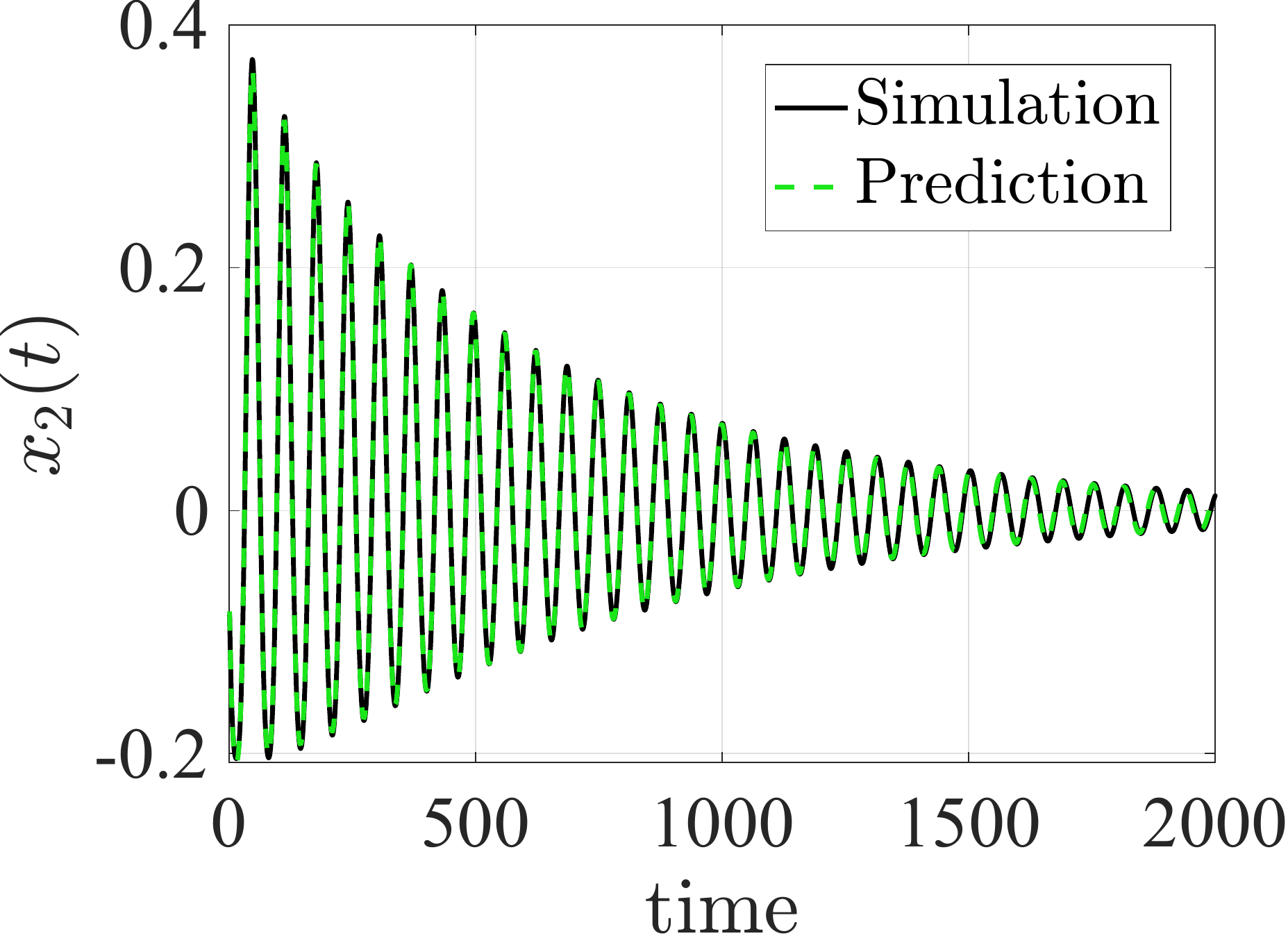}}
	\subfloat[\label{fig:oscssmdelayx1andx2}]{\includegraphics[width=0.33\linewidth]{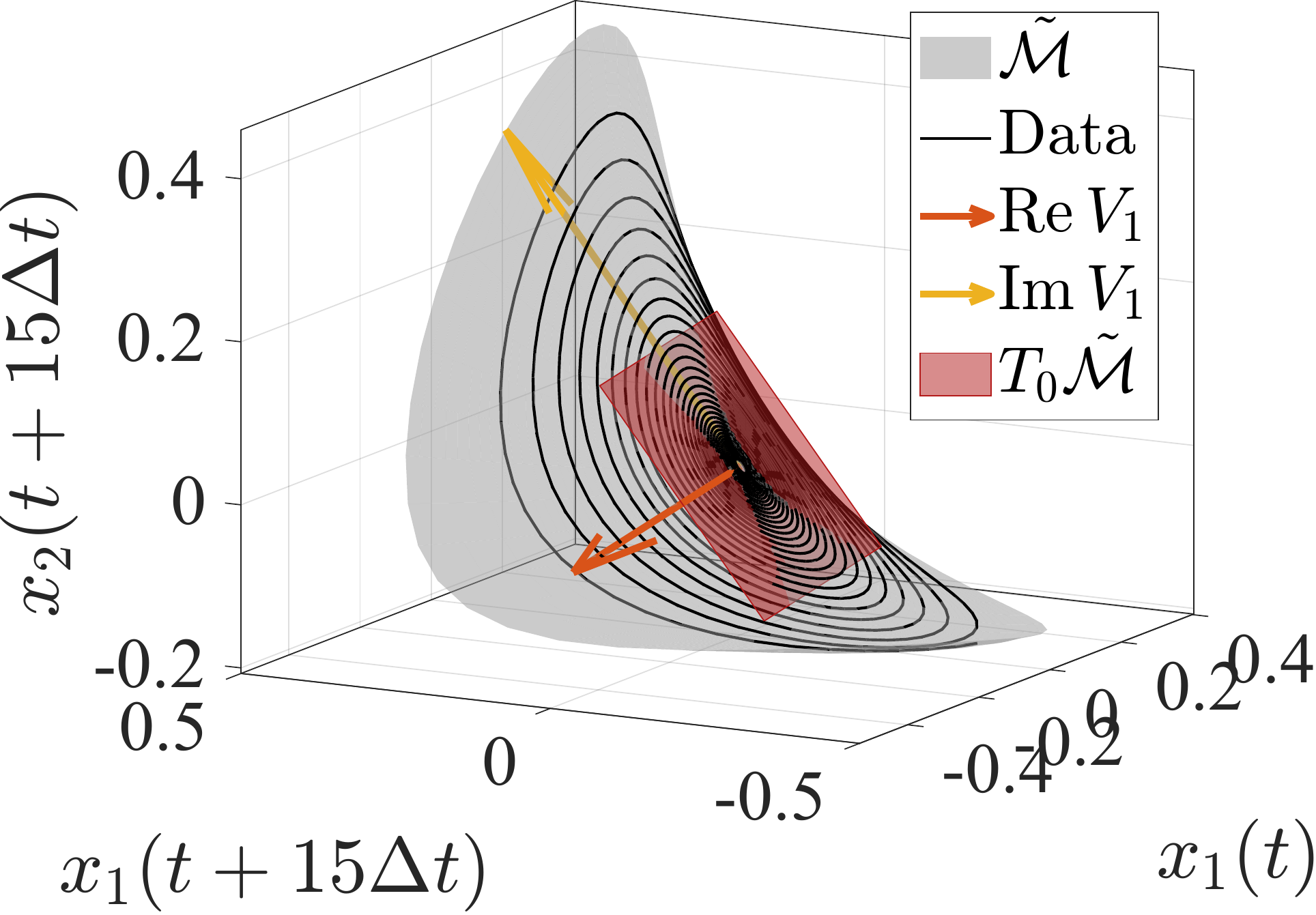}}
	\caption{(a) Projection of the data onto the delay-embedded tangent space predicted by our theory. (b) \mSSM{} predicts a model that successfully reconstructs the decay of the trajectory. (c) A view of the SSM in a delay-embedded space from a multi-dimensional observable, with the tangent space predicted by our theory.}\label{fig:osc}
\end{figure}

Finally, following Theorem \ref{thm:multidim}, we demonstrate how to determine the tangent space of the SSM at the fixed point when the observable is a vector. 
When we choose $\vct \obs(\vct{\fullc}) = [\fullc_1,\fullc_2]^\top$, unlike for a scalar observable function, the tangent space orientation is influenced not only by the eigenvalues, but also by the shape of the first mode.
This first mode shape corresponds to the masses moving in unison, i.e.
\begin{equation}\label{eq:oscmodeshapes}
	\hat{\eigv}_1 = \hat{\eigv}_2 = \left[\begin{array}{c} 1\\ 1 \end{array}\right].
\end{equation}
Then, by (\ref{eq:multidimeig}), we obtain vectors spanning the tangent space as the columns of the matrix
\begin{equation}
	\tang = \left[\begin{array}{cc}
	\vct{\Van} \diag(\hat{\eigv}_1) \\ 
	\vct{\Van} \diag(\hat{\eigv}_2)
	\end{array}\right]
	=
	\left[\begin{array}{cc}
	\vct{\Van} \\ 
	\vct{\Van}
	\end{array}\right],
\end{equation}
where $\vct{\Van}$ is the \van{} \eqref{eq:vandermonde}.
A view of the SSM and its tangent space in this 10-dimensional observable space is shown in Fig.~\ref{fig:oscssmdelayx1andx2}.

The relation of this mode shape to the observable function is given by \eqref{eq:c}.
In particular, we can compute the derivative of the observable function with respect to the modal coordinates as
\begin{equation}
	\left[\begin{array}{cc}
	\frac{\partial \obs_1}{\partial \modc_1}(\vct 0) & \frac{\partial \obs_1}{\partial \modc_2}(\vct 0) \\ 
	\frac{\partial \obs_2}{\partial \modc_1}(\vct 0) & \frac{\partial \obs_2}{\partial \modc_2}(\vct 0)
	\end{array}\right]
	= 
	\left[\begin{array}{cc}
	c_1 & c_2 \\ c_1 & c_2
	\end{array}\right],
\end{equation}
where $c_1,c_2 \in \Ce$ are nonzero constants depending on the scaling of the eigenvectors.
For simplicity, in \eqref{eq:oscmodeshapes} we chose $c_1=c_2=1$, such that $\tang$ is the \van{} vertically stacked twice.

\subsection{6D SSM in a nonlinear finite-element model of a beam}
We train an SSM-reduced model with data from numerical simulations of a finite-element (FE) representation of a clamped-clamped \vk{} nonlinear beam \cite{jain18}.
This example was previously studied in Refs.~\cite{cenedese21,axas22}, which identified the slowest 2D SSM in the delay-embedded observable space, predicted the forced response and analyzed the radius of convergence of the analytical normal form.
Here, thanks to our results on the tangent spaces of delay-embedded SSMs, we can extend the analysis to the six-dimensional SSM emanating from the three slowest modes of the linear part of the system.

Each node in the FE model has three degrees of freedom: axial deformation $u$, transverse deflection $w$, and rotation $w'$. 
The \vk{} axial strain is given by
\begin{equation}
    \epsilon_{11} = u'(x) + \frac{1}{2}\left(w'(x)\right)^2 - zw''(x).
\end{equation}
The axial stress is given by
\begin{equation}
    \sigma = E\epsilon_{11} + c \dot{\epsilon}_{11},
\end{equation}
where $E=70$ GPa denotes the Young's modulus and $c=1.0\times 10^6$ $\mathrm{Pa\cdot s}$ the material rate of viscous damping.
Based on a convergence analysis, we set the number of elements to 12, resulting in a 33-degree of freedom mechanical system, i.e., a 66-dimensional phase space. 
We set the beam length to 1000 mm, width 50 mm, and thickness 20 mm. 
The sampling time is $\Delta t = 0.0955$ s. 

\begin{figure}[h]
	\centering
	\subfloat{\includegraphics[width=0.5\linewidth]{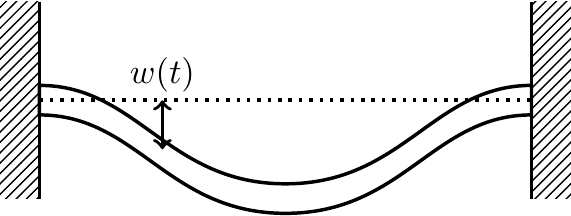}}\\
	\subfloat{\includegraphics[width=0.5\linewidth]{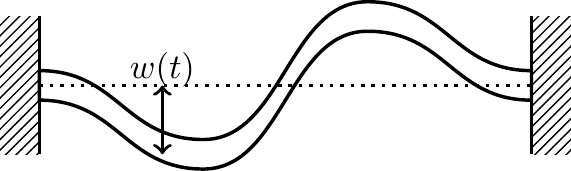}}\\
	\subfloat{\includegraphics[width=0.5\linewidth]{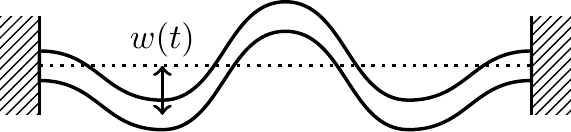}}
	\caption{\vk{} beam: schematic first, second, and third mode shapes. The scalar observable must be generic in the sense that it must have contributions from all modes of interest. For instance, the midpoint displacement is not excited by the second mode. Instead, we choose the shown transverse displacement at $1/4$ of the beam length.}\label{fig:vk}
\end{figure}

By Remark \ref{thm:noobs}, the observable function $\obs$ must have significant contributions from all modes $\modc_\modi$ that we wish to model.
For example, the midpoint displacement chosen as observable function in Refs.~\cite{cenedese21,axas22} was sufficient to model the 2D SSM, but cannot be employed for higher-dimensional SSMs. 
This is because the antisymmetric shape of the second mode has zero displacement at the midpoint (see Fig.~\autoref{fig:vk}), i.e.
\begin{equation}
	\frac{\partial \obs}{\partial \modc_3}(\vct 0) = \frac{\partial \obs}{\partial \modc_4}(\vct 0) = 0.
\end{equation}
Instead, we choose the transverse displacement of the beam at one fourth of the total length, $\obs =w(l/4)$, as this degree of freedom has nonzero contributions from all three mode shapes.

For our data-driven modeling objectives, we need training data containing the first three modes.
To generate initial conditions for such trajectories, we use linear combinations of the mode shapes of the system computed from its linear part.
Since the SSM is normally attracting, these trajectories will quickly approach it and we can use them to train our reduced-order model.
With this method, we produce three trajectories close to the 6D SSM with different initial conditions, of which we use two as training data and one as test data.
For validation purposes, we also pick the individual mode shapes as initial conditions and use as test data.
The individual modal contributions in these initial conditions were chosen as follows:
\begin{center}
\begin{tabular}{|c|ccc|c|}\hline
Initial & & Mode & & Type \\ condition & 1 & 2 & 3 &  \\ \hline
 1 & 1 & 0 & 0 & Test \\ \hline
 2 & 0 & 1 & 0 & Test \\ \hline
 3 & 0 & 0 & 1 & Test \\ \hline
 4 & 0.8 &-0.8 & 0.8 & Train \\ \hline
 5 &-0.1 & 0.8 & 0.8 & Train \\ \hline
 6 &-0.6 &-0.2 &-0.8 & Test  \\ \hline
\end{tabular}
\end{center}

We choose the delay embedding parameters guided by the observations in Sect.~\ref{sec:picktimelag}.
Setting $\lag=1$ such that the timelag $\tau=\Delta t$ and the embedding dimension to $\deldim=50$ gives a local optimum of the function (\ref{eq:delaycost}) with the computed eigenvalues, while still keeping the maximal delay $\lag\deldim$ moderate to prevent folding of the embedding.

Fig.~\ref{fig:vkdelayspaces} shows the delay embedding of the single-mode trajectories 1-3, corresponding to the first three modes, in three of the 50 delay coordinates.
These trajectories visualize the orientations of the corresponding eigenspaces in the observable space. 
Indeed, minimization of (\ref{eq:delaycost}) corresponds to making these planes orthogonal, simplifying their identification.

\begin{figure}[h]
	\centering
	\subfloat[\label{fig:vkdelayspaces}]{\includegraphics[width=0.33\linewidth]{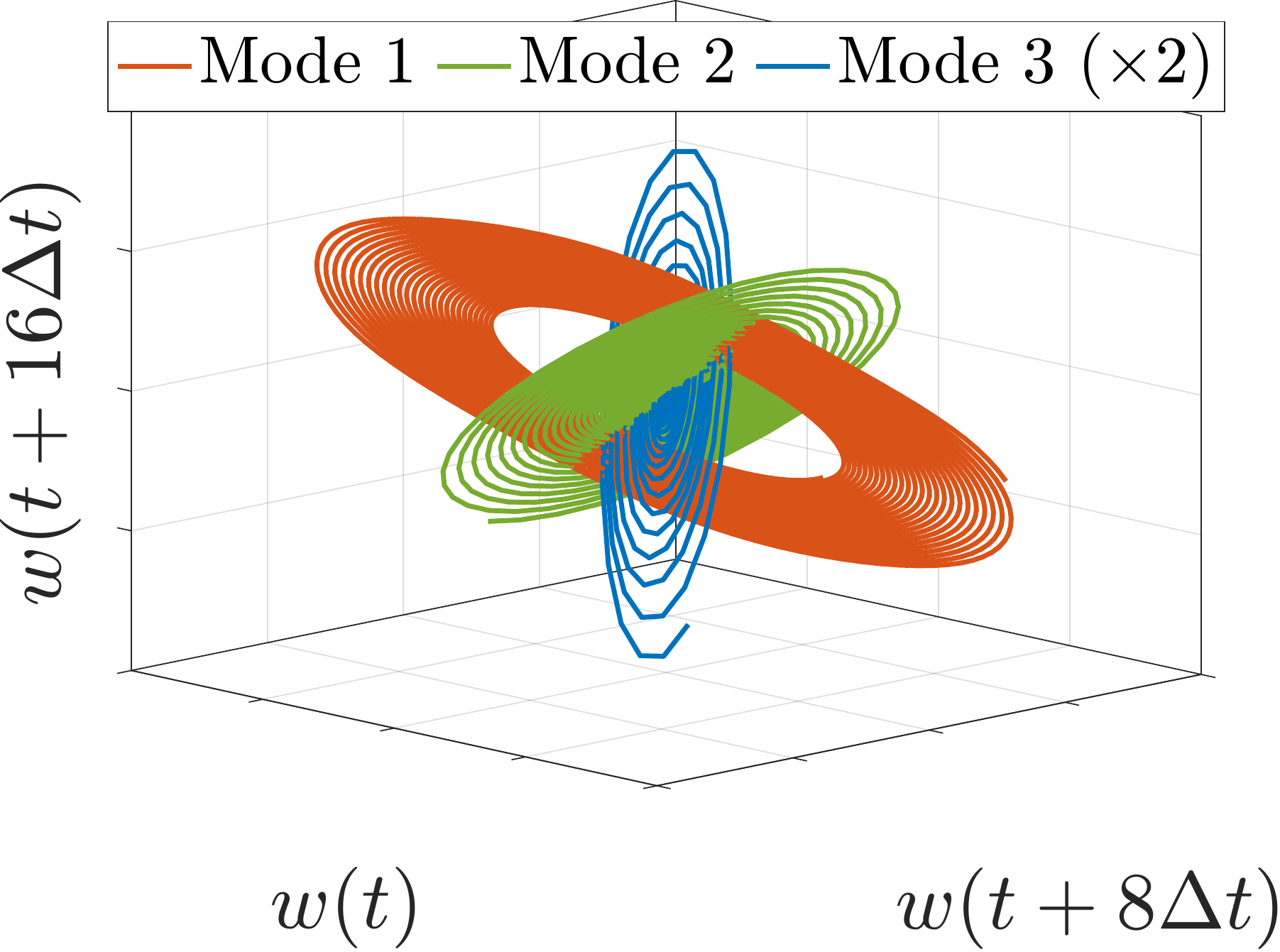}}\hfill
	\subfloat[\label{fig:vkdelaytrain}]{\includegraphics[width=0.33\linewidth]{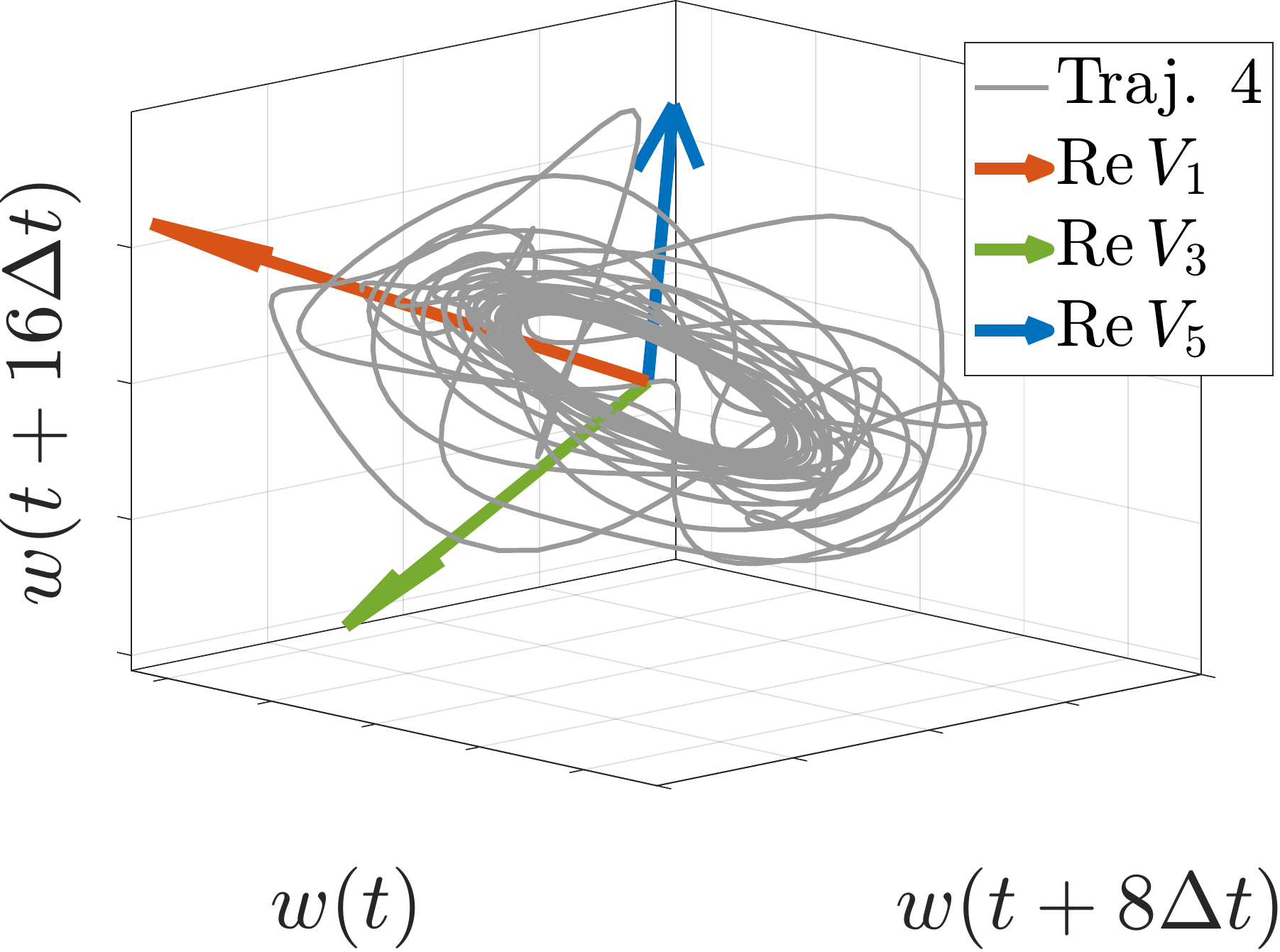}}\hfill
	\subfloat[\label{fig:vkred}]{\includegraphics[width=0.33\linewidth]{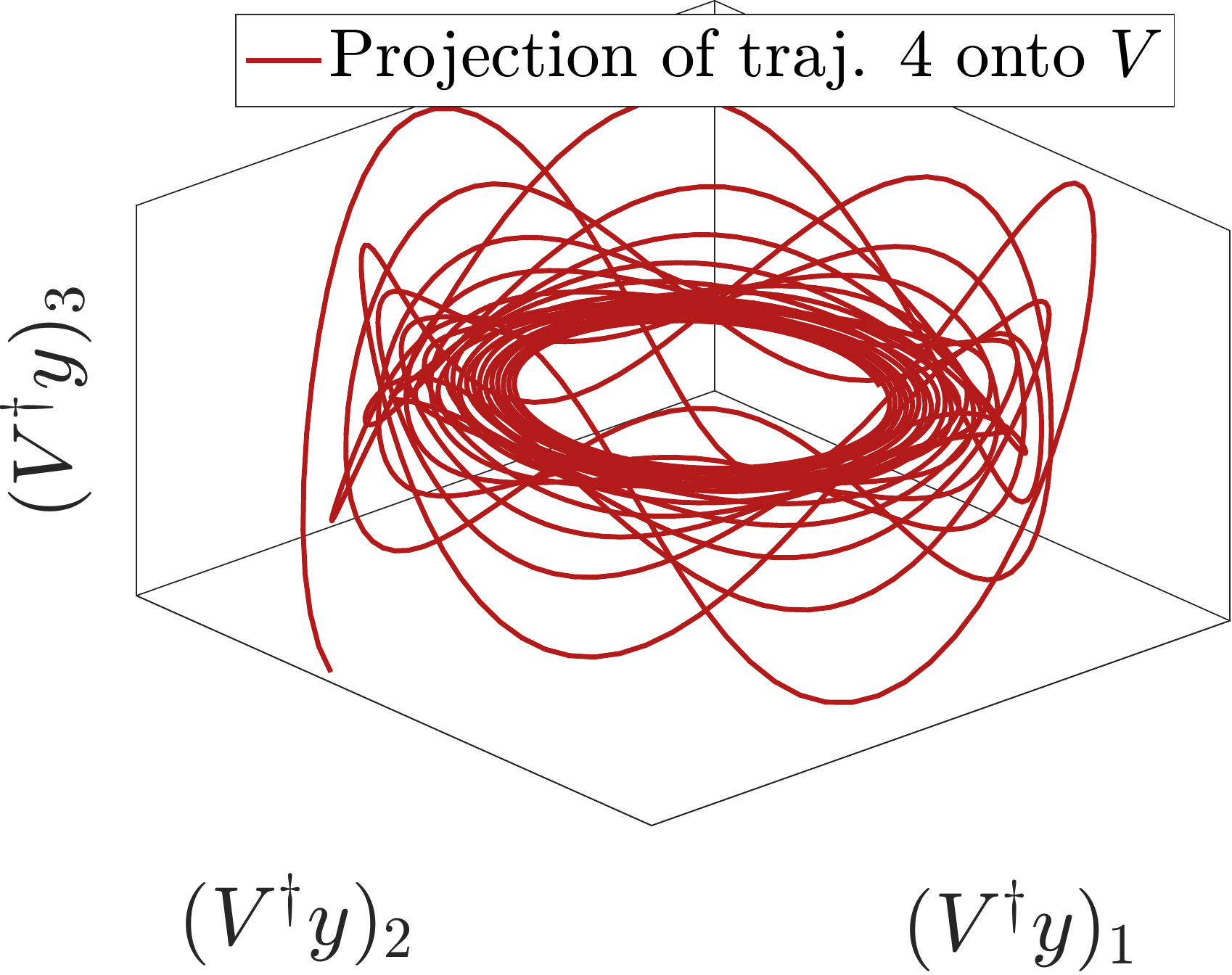}}
	\caption{(a) The trajectories with single modal contributions visualize the modal subspaces in the delay-embedded space. The third mode data has been scaled by a factor 2 to increase visibility. (b) The same delay-embedded view of the first training trajectory, along with the delay-embedded eigenvectors. (c) After projection of this trajectory onto the eigenvectors, the modal structure becomes clear.}\label{fig:vkdelay}
\end{figure}

Fig.~\ref{fig:vkdelaytrain} similarly displays the delay embedding of the first training trajectory along with a visualization of the columns of the \van{} as vectors. 
Our delay theory predicts that projection of the data onto these vectors yields modal coordinates, as shown in Fig.~\ref{fig:vkred}.
This space will serve as the reduced coordinates of the SSM.

After projection onto these eigenvectors, we approximate the geometry of the 6D SSM with a 3rd order polynomial.
For the reduced dynamics in \mSSM{}, we also use a 3rd order approximation. 
We compute the normal form of this reduced dynamics up to 7th order and obtain our model for the reduced dynamics.
The terms up to third order in polar form are found by \mSSM{} to be of the form
\begin{equation}\label{eq:vknf}
	\left(\begin{array}{c} \dot{\rho}_{1} \\ \rho_{1}\dot{\theta}_{1} \\ \dot{\rho}_{2} \\ \rho_{2}\dot{\theta}_{2} \\ \dot{\rho}_{3} \\ \rho_{3}\dot{\theta}_{3} \\  \end{array}\right)
	=
	\left(\begin{array}{c} 
	0.3058\,{\rho _{1}}^3+2.088\rho _{1}\,{\rho _{2}}^2-3.091\rho _{1}\\ 
	102.0\,{\rho _{1}}^3+82.70\,{\rho _{2}}^2\rho_1+657.2\rho_1\\ 
	-2.705\,{\rho _{1}}^2\rho _{2}+1.723\,{\rho _{2}}^3-23.72\rho _{2}\\ 
	95.64\,{\rho _{1}}^2\rho_2+115.6\,{\rho _{2}}^3+1812\rho_2\\ 
	-8.968\rho _{3}\,{\rho _{1}}^2-13.27\rho _{3}\,{\rho _{2}}^2-88.47\rho _{3}\\ 
	115.9\,{\rho _{1}}^2\rho_3+85.04\,{\rho _{2}}^2\rho_3+3558\rho_3 \end{array}\right)+\ordo{5}.
\end{equation}

We transform the initial conditions from the observable space to the normal form and integrate our model to predict signal decay.
This produces a normalized mean trajectory error (as defined in \cite{cenedese21}) of 2.2 \% on the test data.
Some of the predictions are shown in Fig.~\ref{fig:vkdecay}.

\begin{figure}[h]
	\centering
	\subfloat[\label{fig:vkdecaytrain}]{\includegraphics[width=0.33\linewidth]{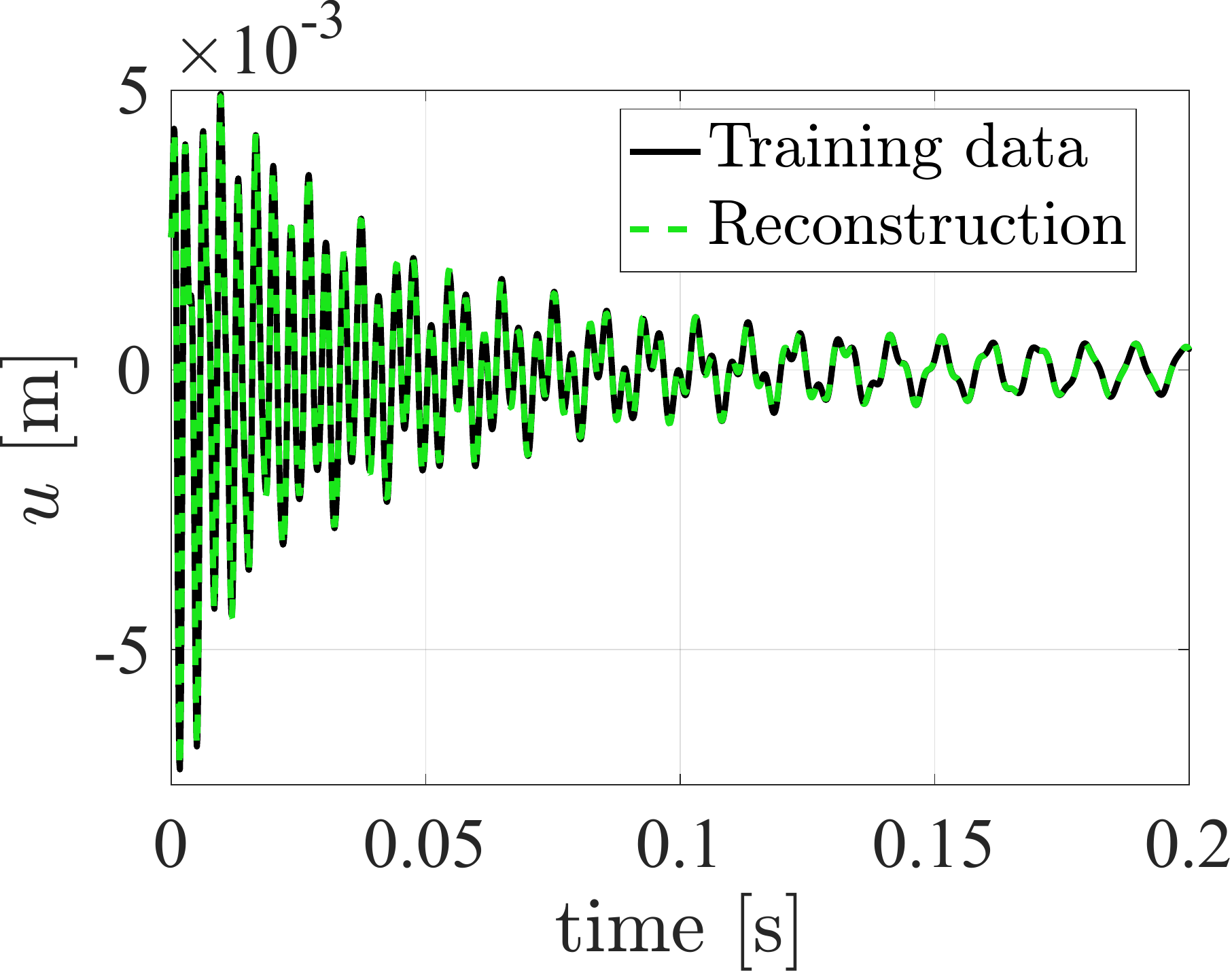}}
	\subfloat[\label{fig:vkdecaytest}]{\includegraphics[width=0.33\linewidth]{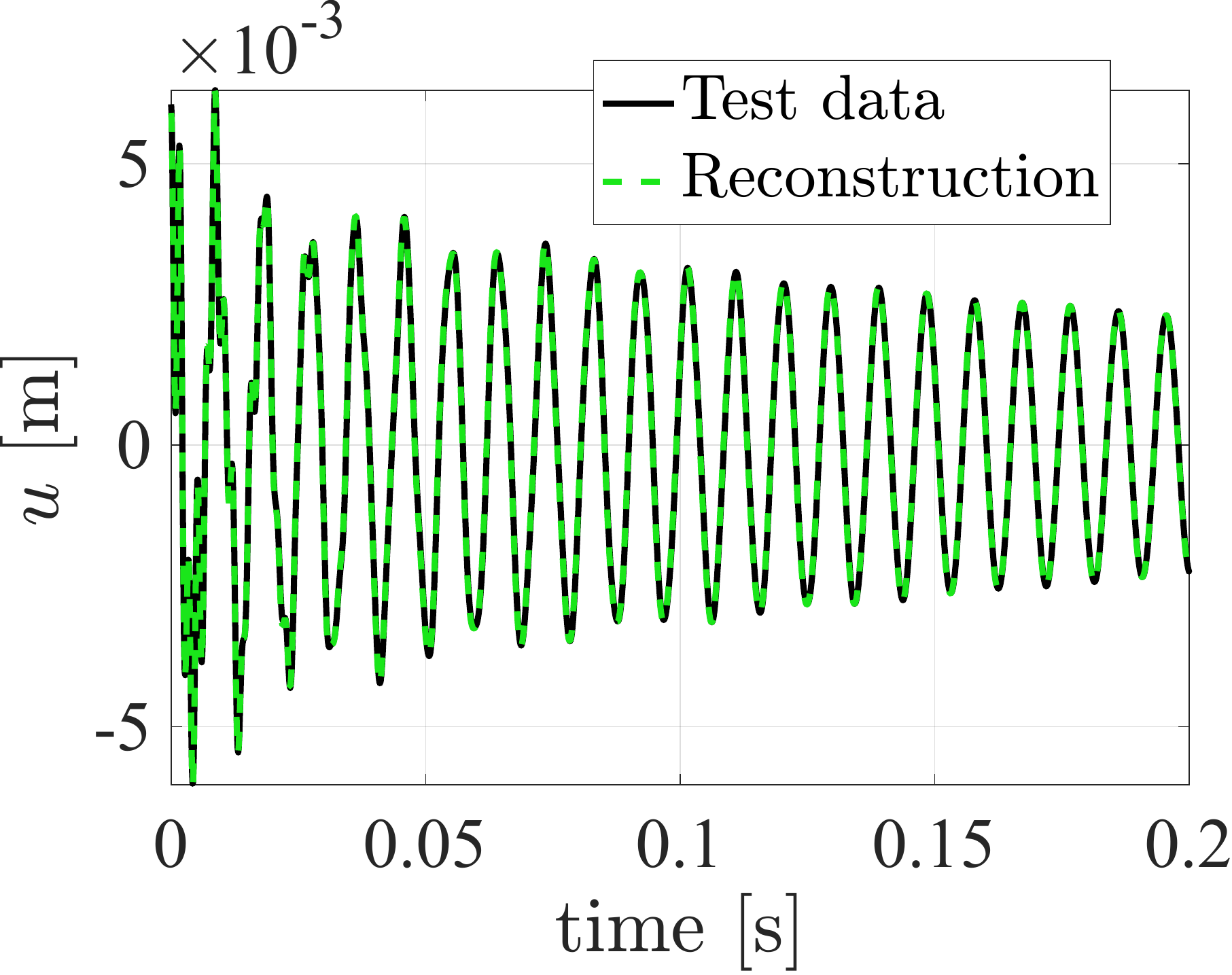}}
	\subfloat[\label{fig:vkphaseportrait}]{\includegraphics[width=0.33\linewidth]{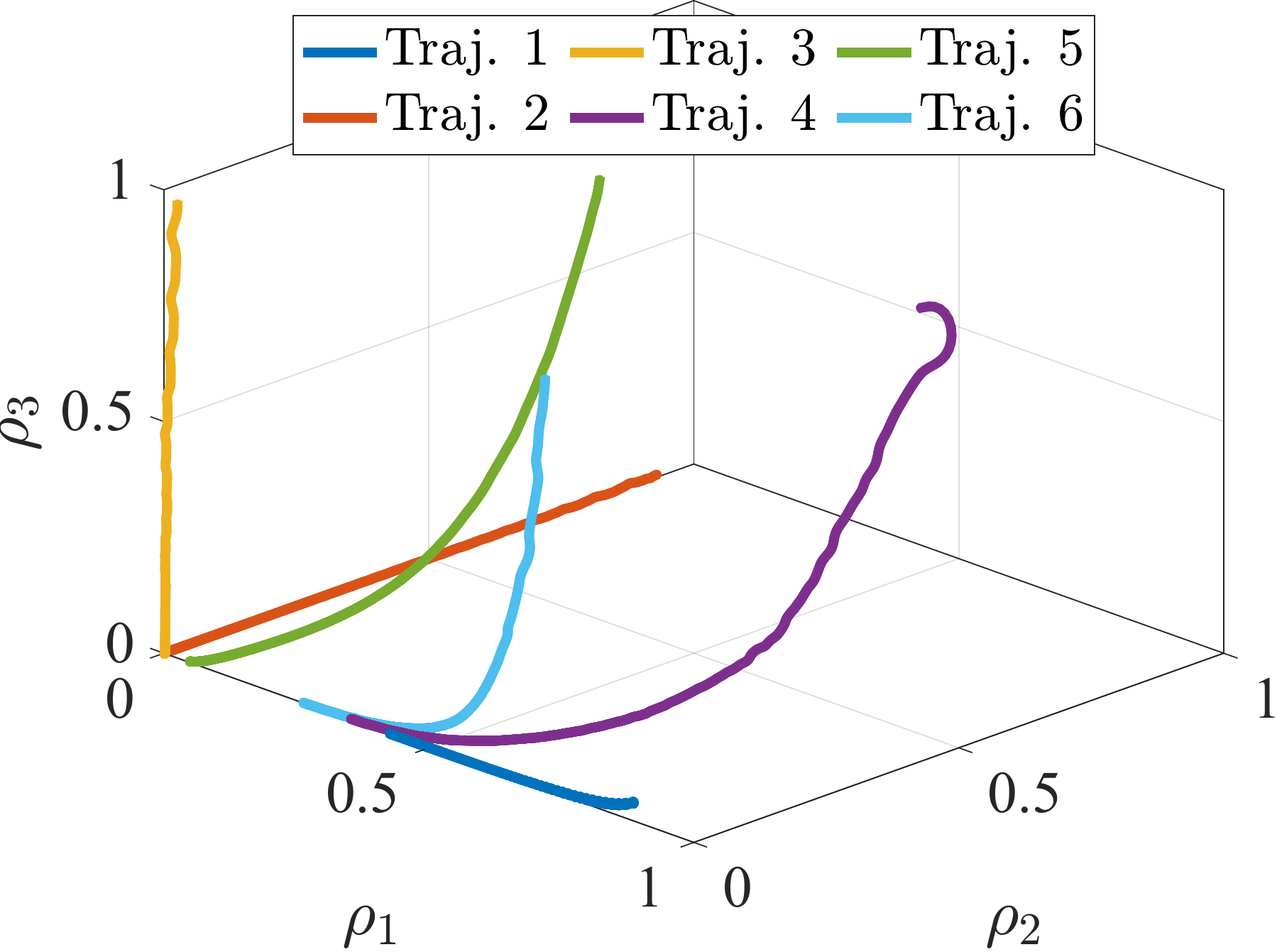}}
	\caption{(a,b) Predictions from \mSSM{} for the decaying trajectories 5 and 6 (c) Phase portrait of the trajectories after transformation to the normal form.}\label{fig:vkdecay}
\end{figure}

We also visualize our reduced-order model by plotting the instantaneous frequency and damping as predicted by the normal form \eqref{eq:vknf} for varying amplitudes of mode 1 and 2. 
For instance, our model predicts hardening of the first mode with respect to both the first and the second modal amplitudes (Fig.~\ref{fig:vkeigenfreq1}), a decrease in the instantaneous damping of mode 1 with respect to mode 2 (Fig.~\ref{fig:vkeigendamp1}), and independence of the third instantaneous frequency with respect to itself (Fig.~\ref{fig:vkeigenfreq3}).
The predictions for each of the trajectories are included for reference.

\begin{figure}[h]
	\centering
	\subfloat[\label{fig:vkeigenfreq1}]{\includegraphics[width=0.33\linewidth]{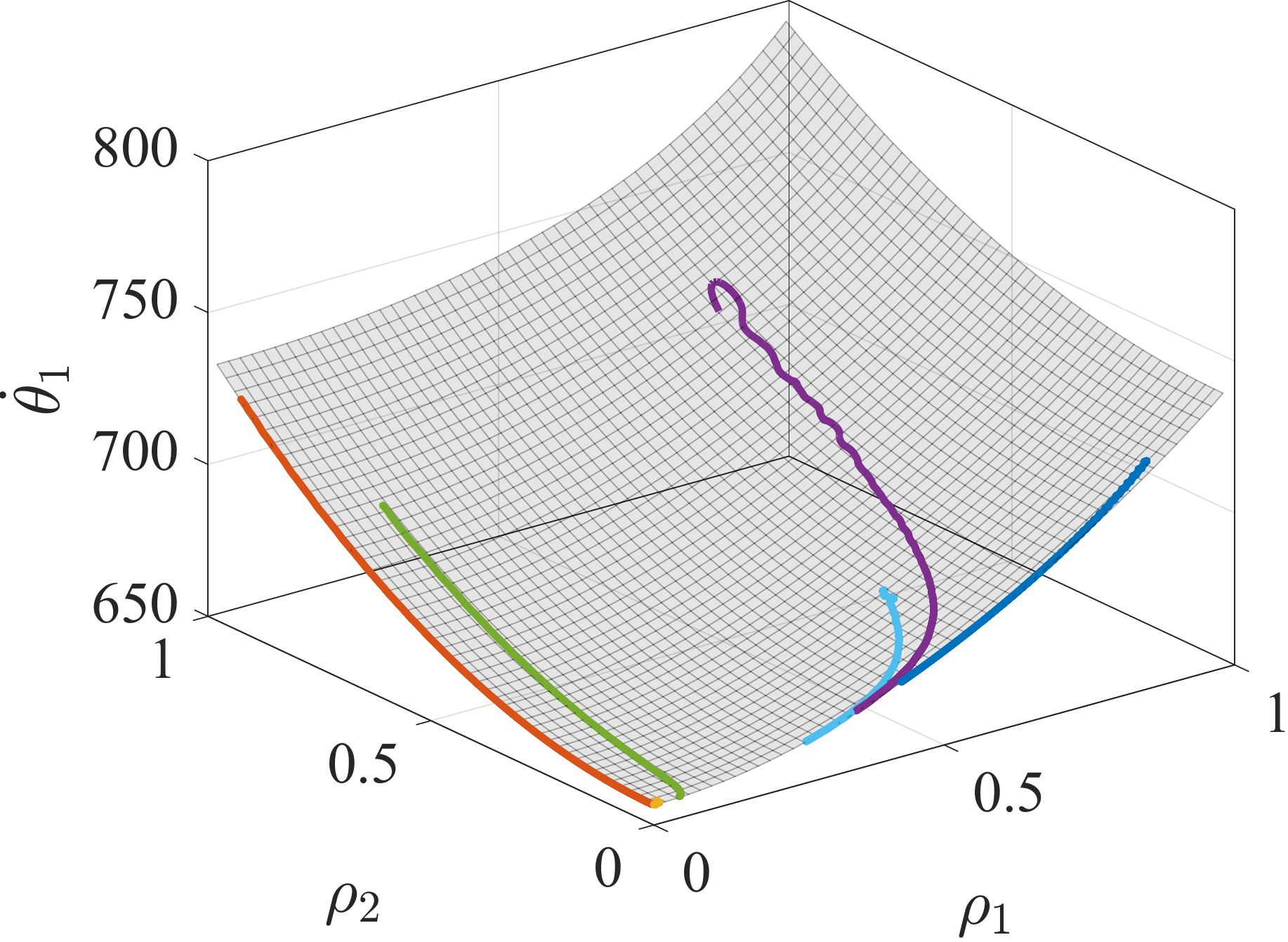}}\hfill
	\subfloat[\label{fig:vkeigendamp1}]{\includegraphics[width=0.33\linewidth]{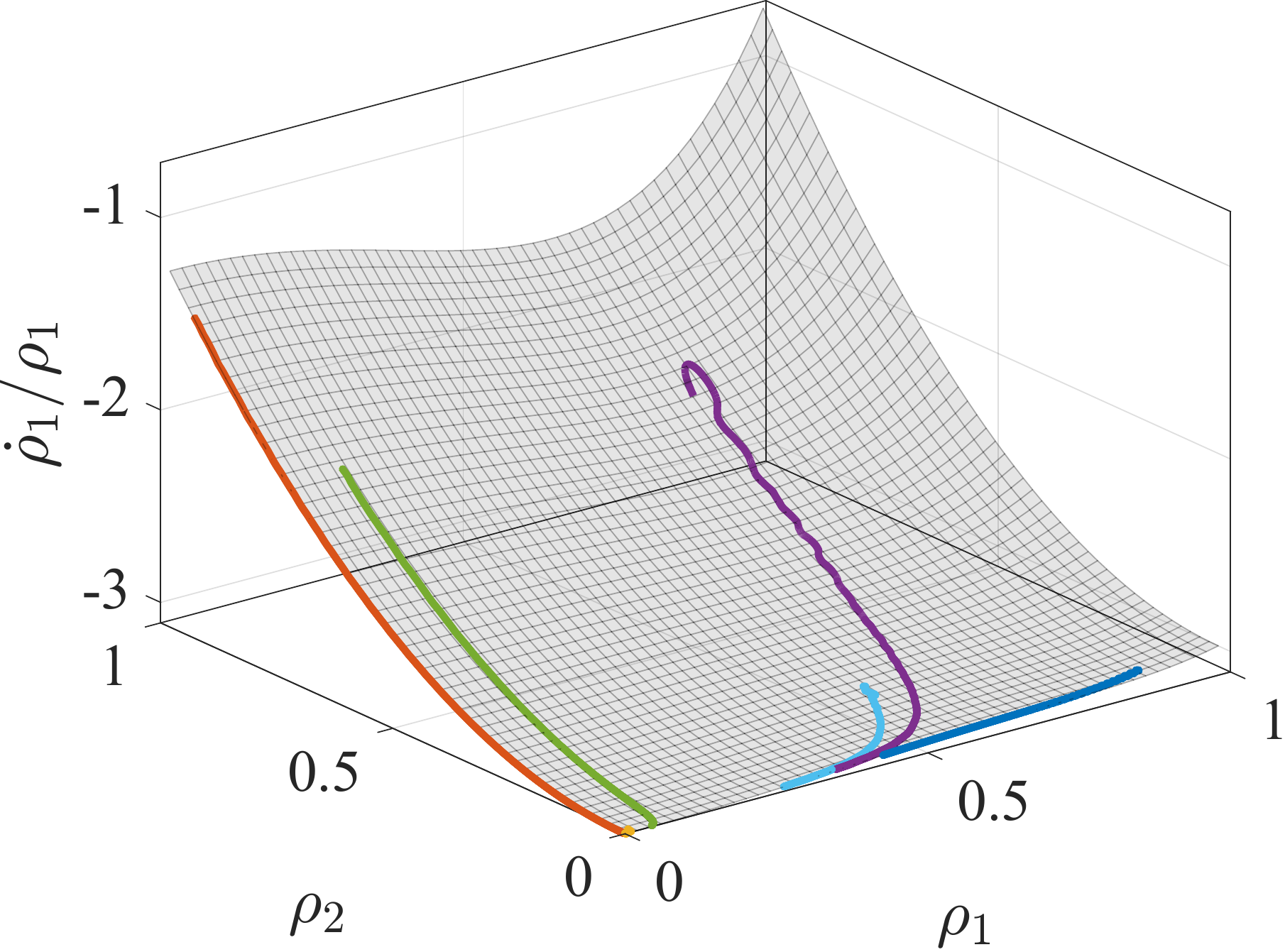}}
	\subfloat[\label{fig:vkeigenfreq3}]{\includegraphics[width=0.33\linewidth]{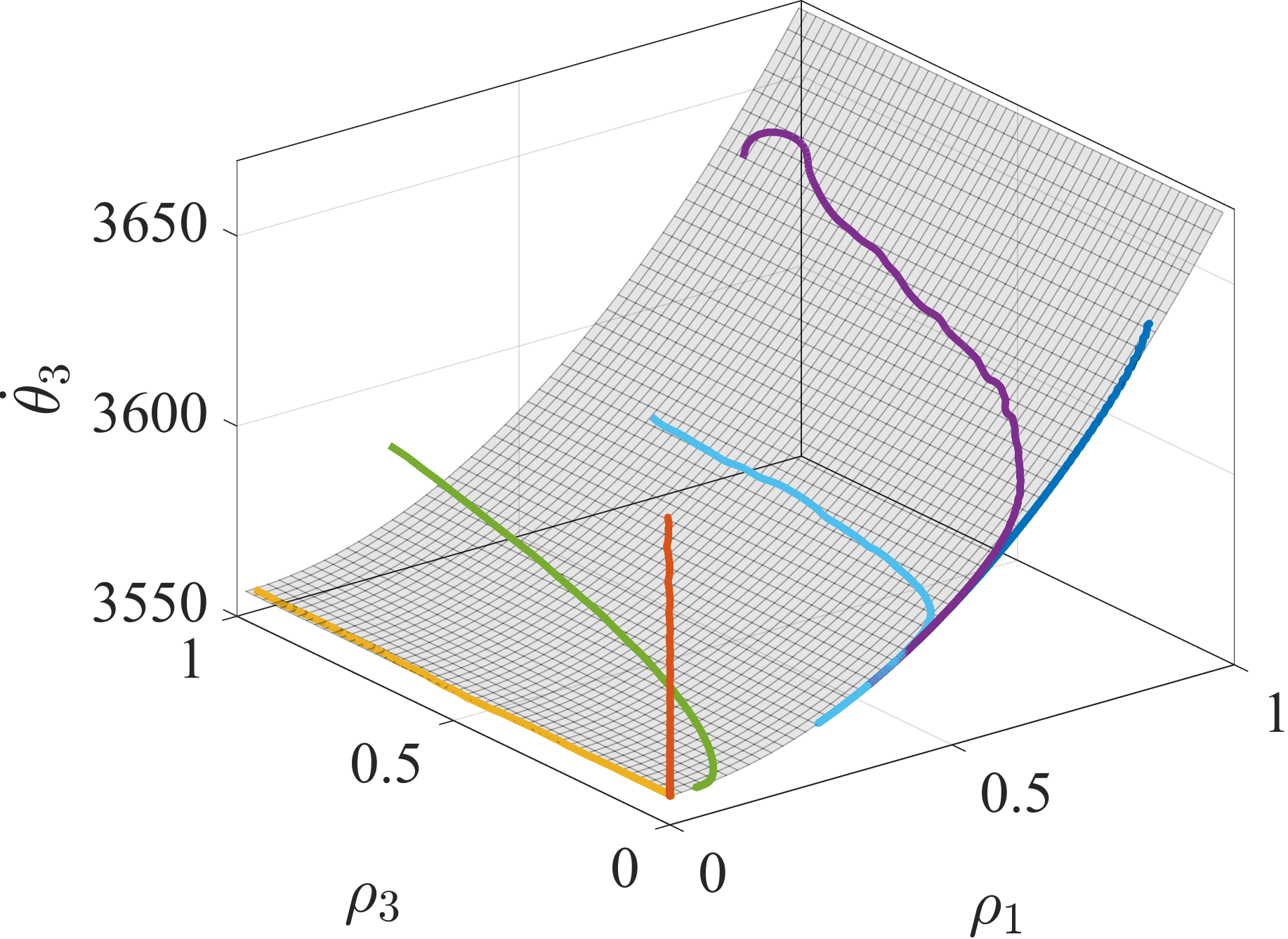}}
	\caption{Visualization of the normal form \eqref{eq:vknf} with the trajectories for (a) instantaneous frequency and (b) damping of mode 1, as well as (c) frequency of mode 3.}\label{fig:vknf}
\end{figure}

\subsection{Multimodal sloshing of water in a tank}

\begin{figure}[h]
	\centering
	\subfloat[\label{fig:sloshingsetup}]{\includegraphics[width=0.65\linewidth]{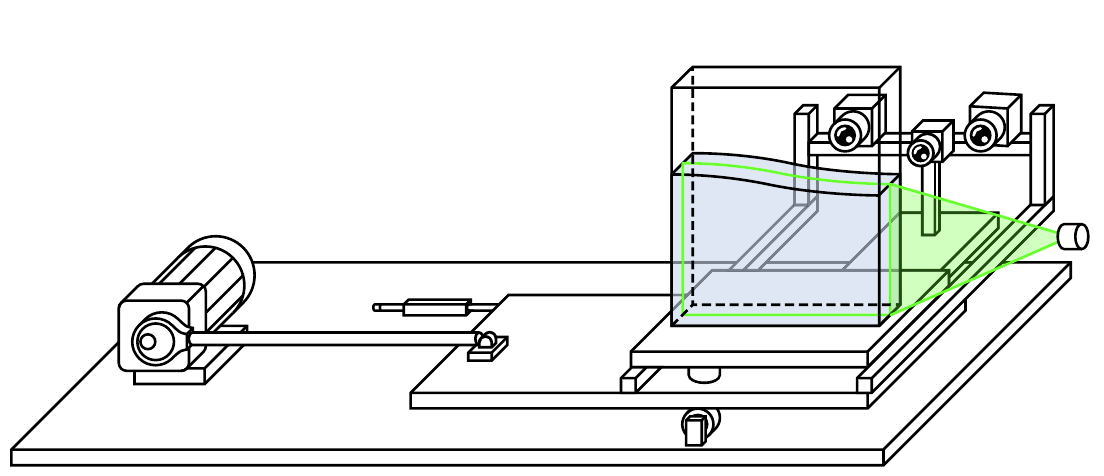}}\hfill
	\subfloat[\label{fig:sloshingmodeshapes}]{\includegraphics[width=0.35\linewidth]{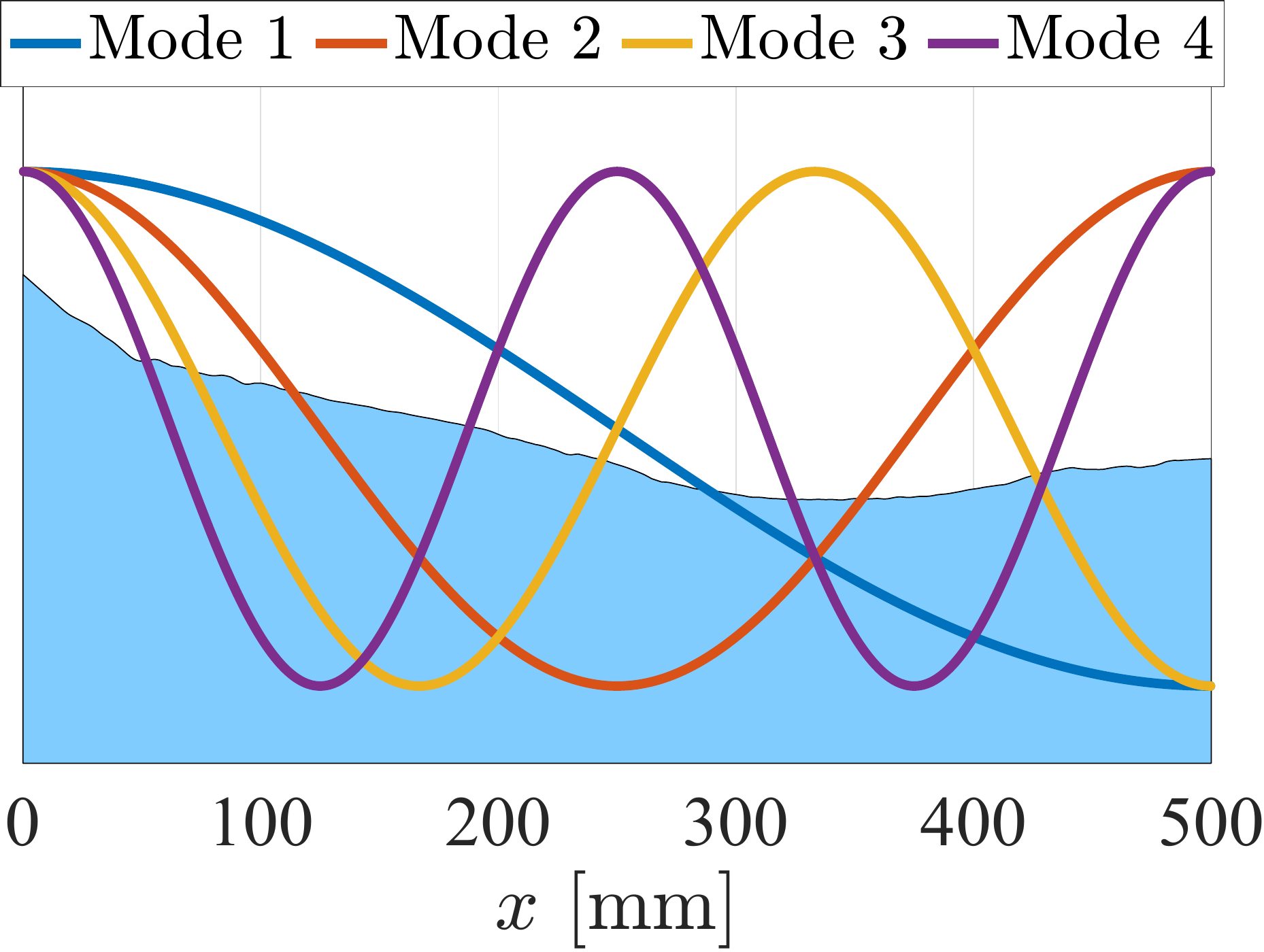}}
	\caption{(a) Experimental setup for tank sloshing (adapted from \cite{bauerlein21}) (b) The first four sloshing mode shapes.}\label{fig:sloshingprofile}
\end{figure}

For our final example, we apply our results to sloshing experiments.
Sloshing models have a wide range of industrial applications, including fluid container interaction with ship motion \cite{mitra12}, road transportation of fluids \cite{fleissner10}, damping devices in towers \cite{hickey22}, and fuel tank design in spacecraft \cite{dodge00,abramson66}.
A tank partially filled with water exhibits several nonlinear phenomena under horizontal harmonic excitation \cite{taylor53}.
On the one hand, intensified fluid motion can alter the instantaneous damping and frequency of the first sloshing mode \cite{faltinsen09}.
On the other hand, increasing the amplitude further activates several nonlinearly coupled modes of the system and gives rise to a range of different wave motions \cite{narimanov57,faltinsen00}.

Our training data comes from experiments described in Ref.~\cite{bauerlein21} with a rectangular tank of width $w=500$ mm and thickness $50$ mm, partially filled with water up to a height of $h=400$ mm. 
The tank was attached to a horizontally moving platform harmonically excited by a motor at different frequencies. 
Then, once the system had reached a steady state, the motor was turned off.
Depending on the forcing frequency, this periodic response exhibited planar, wave-breaking, or three-periodic motion. 
The three-periodic forced state was characterized by an increase in the response amplitude every third forcing cycle, while the wave-breaking response was defined as overturning of the water close to the walls \cite{bauerlein21}.
A camera detected the surface profile $\vct h$ with the sampling time $\Delta t = 0.01$ s. 
Figure \ref{fig:sloshingsetup} displays the experimental setup. 

While previous work successfully captured the dynamics of the main sloshing mode using a 2D SSM for the center of mass signal \cite{cenedese21} and the full surface profile \cite{axas22}, here, we model the decay from a multimodal state by identifying a 6D SSM, corresponding to the nonlinear extension of the three dominant oscillatory modes.
We train on three decaying measurements: Trajectory 1 and 2 start at a three-periodic state, and Trajectory 3 starts at a wave-breaking state. 

The observable vector $\vct \obs$ is the surface profile measured at $1\,771$ points along the tank width. 
Since this function is multi-dimensional, in order to apply our theory on delay-embedded tangent spaces, we need an estimate of the eigenvalues and linear mode shapes in our observable. 
The eigenfrequencies can be computed from potential theory \cite{faltinsen09} as
\begin{equation}
	\omega_{\modi} = \sqrt{\frac{g\pi}{w}\modi \tanh\left(\pi\modi\frac{h}{w}\right)},
\end{equation}
which scales approximately with the square root of the mode number $\modi$ for our configuration. 
The first five eigenfrequencies are [7.80, 11.1, 13.6, 15.7, 17.6]~rad/s, with an approximate 1:2 resonance between frequencies 1 and 4.
The mode shapes by the same theory are
\begin{equation}
	\hat{\vct \eigv}_\modi = \cos(\modi x/w),\quad x\in [0,w],
\end{equation}
shown in Fig.~\ref{fig:sloshingmodeshapes}.
For the tangent space, in principle, we also need the linear damping of each mode.
In practice, this real part of the eigenvalues has very little influence on $\vct{\Van}$ for limited delay embedding and we pick the values $[-0.05,-0.07,-0.08,-0.09,-0.1]$ based on previous fits of the first mode and the assumption of increasing damping with the mode number.

\begin{figure}[h]
	\centering
	\subfloat[\label{fig:sloshingred}]{\includegraphics[width=0.33\linewidth]{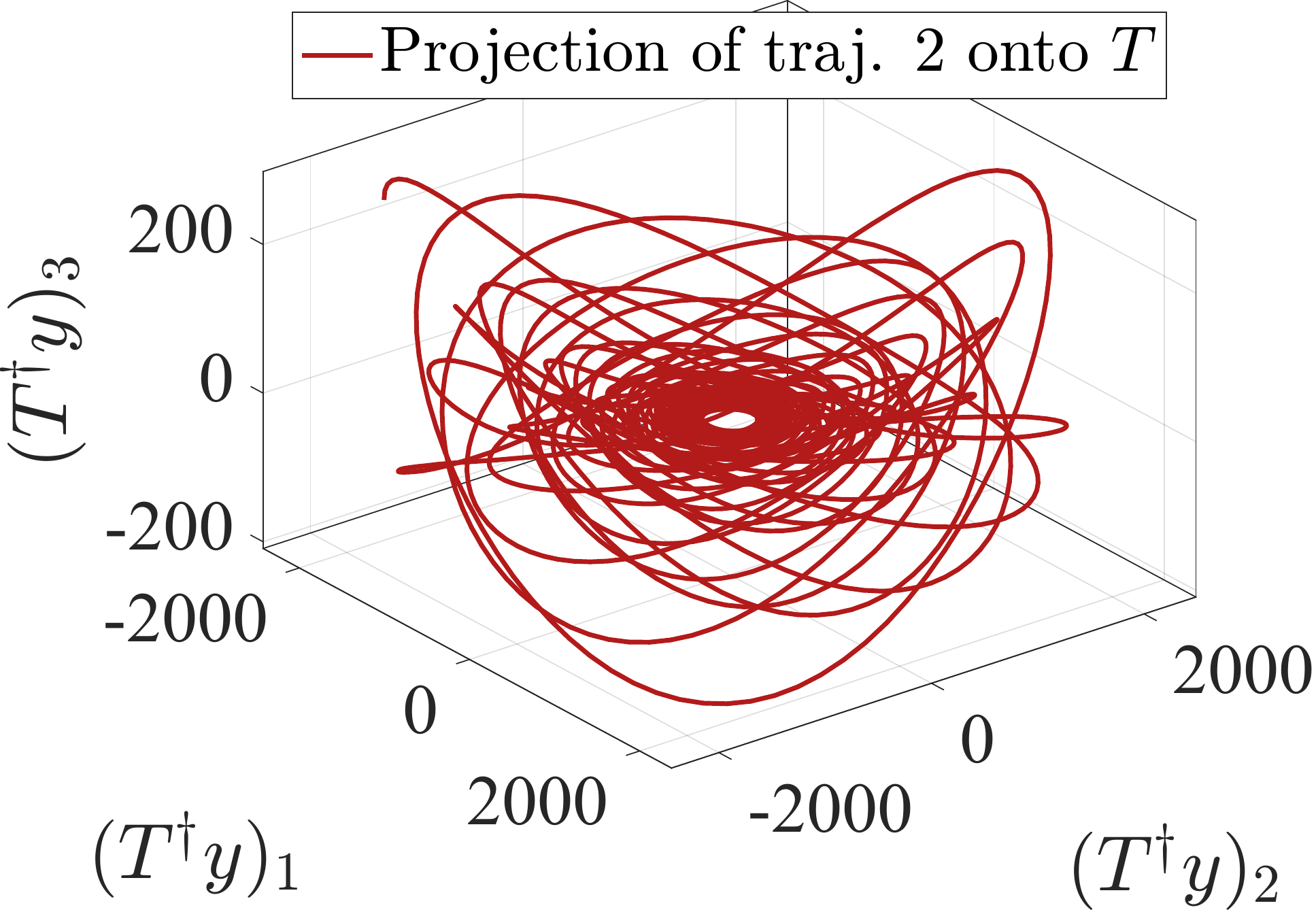}}\hfill
	\subfloat[\label{fig:sloshingmodalamps}]{\includegraphics[width=0.33\linewidth]{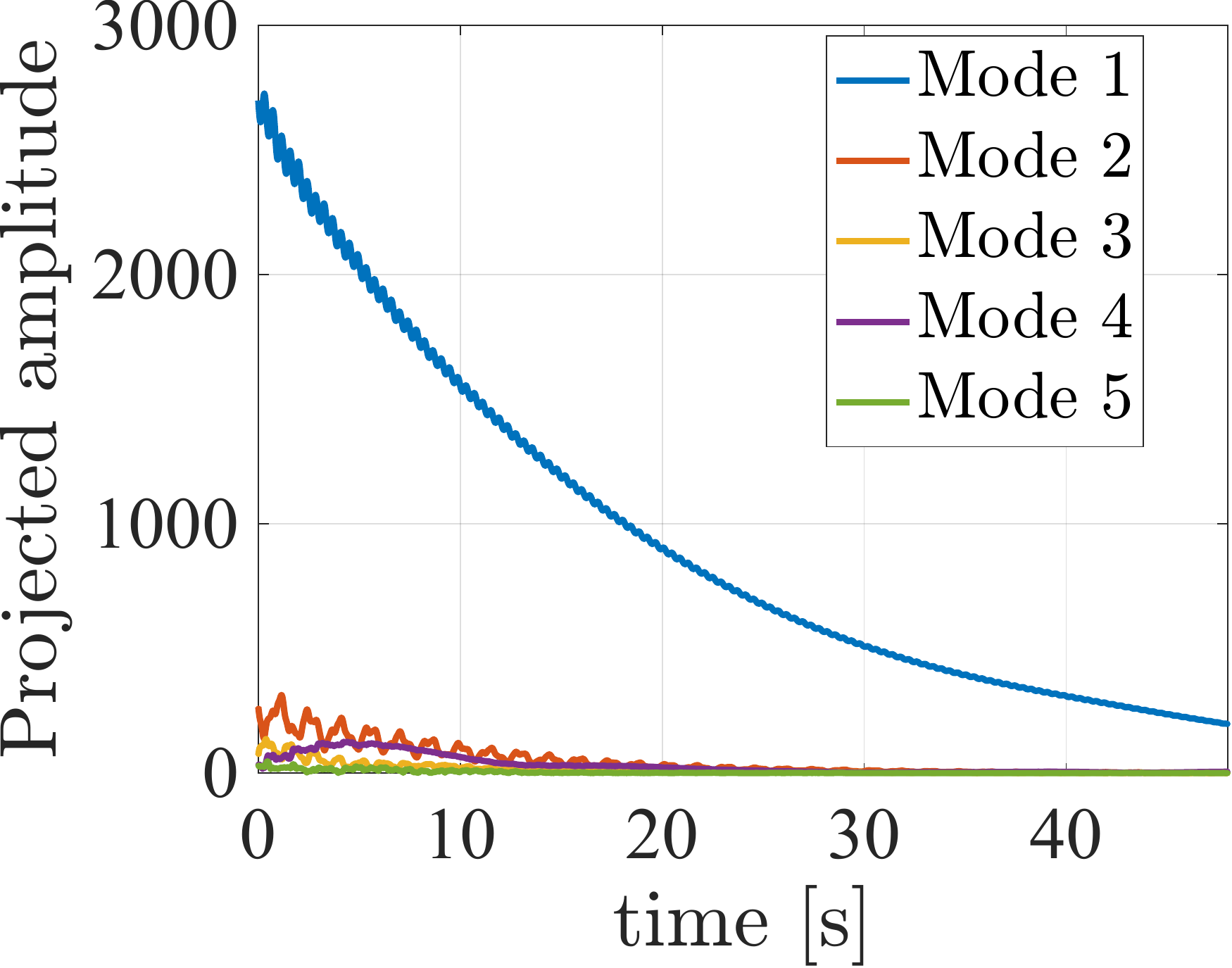}}\hfill
	\subfloat[\label{fig:sloshingmodalampszoom}]{\includegraphics[width=0.33\linewidth]{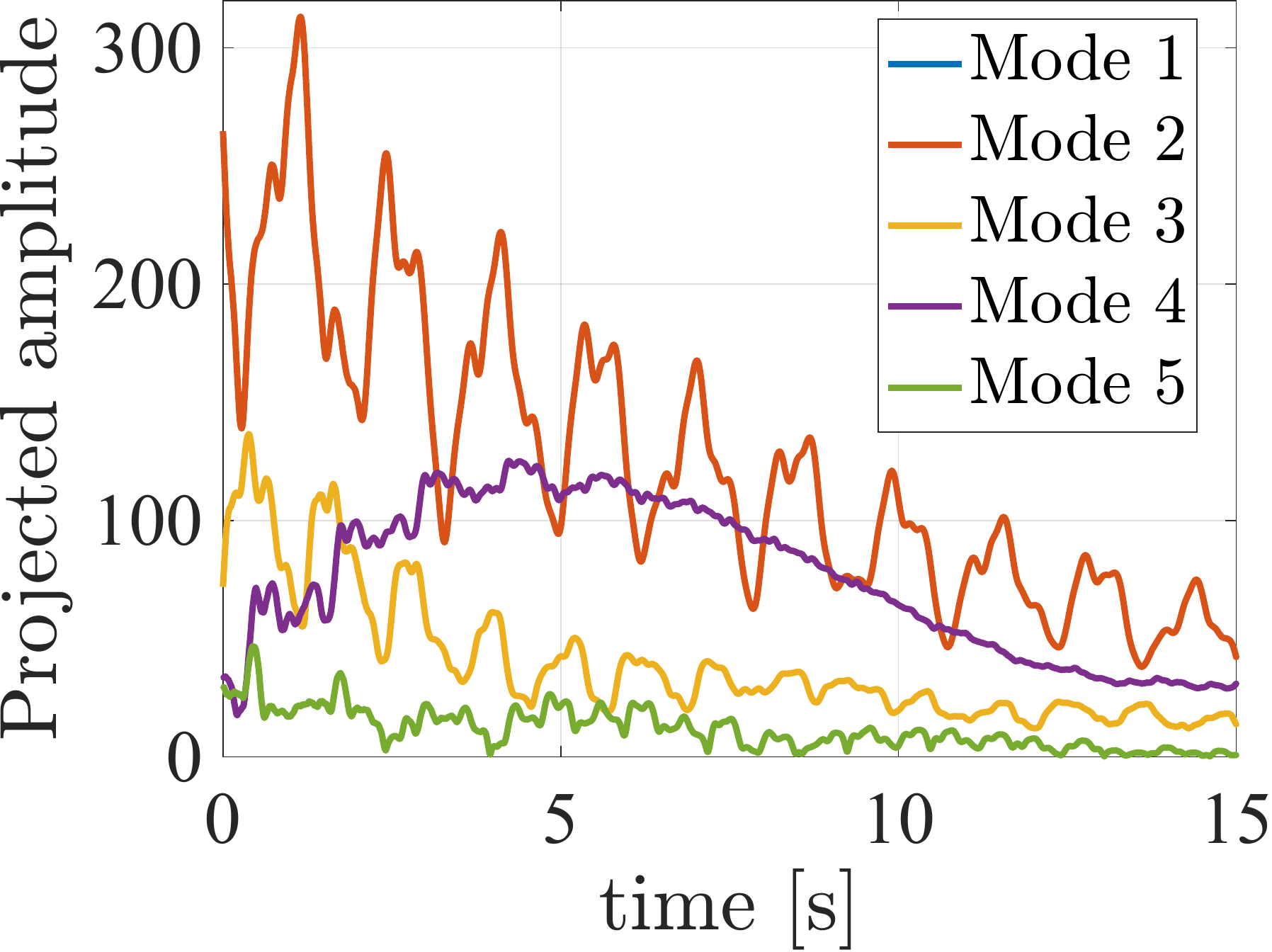}}
	\caption{(a) Projecting one of the trajectories onto the tangent space vectors unveils the modal structure. (b) By projecting the trajectory onto the eigenvectors and taking the absolute value, we can estimate the relative modal contributions in the signal (c) A zoomed-in view of (b) indicates that modes 1, 2, and 4 dominate.}\label{fig:sloshingdelayanalysis}
\end{figure}

Based on (\ref{eq:delaycost}), we delay-embed the data with timelag $\tau=5\Delta t$ and dimension $\deldim=47$.
A projection of the delay-embedded data onto the eigenvectors $\tang$ predicted by our theory appears to yield modal coordinates, as indicated in Fig.~\ref{fig:sloshingred}.

Consequently, the norm of these projections can be used as a heuristic measure of the modal content in the signal.
This procedure should be used with caution, since it does not take manifold curvature into account, but it can be employed to provide an initial guess for the SSM dimension. 
In Fig.~\ref{fig:sloshingmodalamps}, we plot the absolute value of the projection onto each modal subspace of $\tang$ over time for Trajectory 2.
This plot shows that the first mode dominates, while the zoomed-in view (Fig.~\ref{fig:sloshingmodalampszoom}) indicates that the second and fourth modes appear to be the most prevalent of the higher modes throughout the decay.
The third and fifth mode are present at first but quickly die out.
All amplitudes are decaying except for the fourth mode, which instead initially grows.
Based on this analysis, we will identify a 6D SSM emanating from the spectral subspace of modes 1, 2, and 4. 
This choice also takes SSM theory into account, by which the 1:2 resonance requires that the modal subspace of the fourth mode is included in the spectral subspace of the SSM.
We choose to start our training data after 1.2~s, as the third and fifth modal amplitudes are small thereafter and we expect the trajectory to lie sufficiently close to the SSM.

With an SSM parametrization order $\ssmorder=4$, reduced dynamics order $\redorder=3$, and normal form order $\nforder=3$, we compute the SSM geomety and dynamics and integrate our reduced-order model to predict the decay from the various flow states. 
This yields a normalized mean trajectory error (NMTE) \cite{cenedese21} of 2.6 \%.
\mSSM{} successfully detects and accounts for the internal resonance by adding phase-dependent terms to the computed normal form, which reads
\begin{equation}
	\begin{array}{l}
	\frac{\dot{\rho}_{1}}{\rho_1} =
	-0.056-0.0069\sin(\psi-0.26)\rho_4-0.0015\rho_4^2-0.039\rho_{2}^2+0.023\rho_{1}^2\\ 
	\dot{\theta}_{1} =
	7.78+0.0069\cos(\psi-0.26)\rho_4+0.040\rho _{4}^2+0.016\rho _{2}^2-0.82\rho_{1}^2\\ 
	\frac{\dot{\rho}_2}{\rho_2} =
	-0.13+0.15\rho_4^2-0.89\rho_{2}^2+0.37\rho_{1}^2\\ 
	\dot{\theta}_{2} =
	11.4+0.57\rho_{4}^2-0.0085\rho_{2}^2-2.2\rho_{1}^2\\ 
	\frac{\dot{\rho}_{4}}{\rho_4} =
	-0.30-0.29\rho_4^2+0.67\rho_{2}^2-0.27\sin(\psi+1.4)\rho_{1}^2+1.2\rho_{1}^2\\ 
	\dot{\theta}_{4} =
	15.9-0.085\rho_4^2+1.2\rho_{2}^2+0.27\cos(\psi+1.4)\rho_{1}^2-2.0\rho_{1}^2 
	\end{array}
\end{equation}
where $\psi=\theta_4-2\theta_1$ and the subscripts denote the corresponding mode number.
Looking at the linear part, we see that the eigenfrequencies are well captured.

Good agreement between the experimentally measured surface profile elevation at the tank's leftmost point and the delay-embedded SSM-reduced prediction is shown for the first period-three initial state in Fig.~\ref{fig:sloshingdecay1} and the wave-breaking state in Fig.~\ref{fig:sloshingdecay3}.
Further, our 6D reduced model can accurately predict the full surface profile decay, with snapshots shown in Figure \ref{fig:sloshingprofile123}.

\begin{figure}[h]
	\centering
	\subfloat[\label{fig:sloshingdecay1}]{\includegraphics[width=0.33\linewidth]{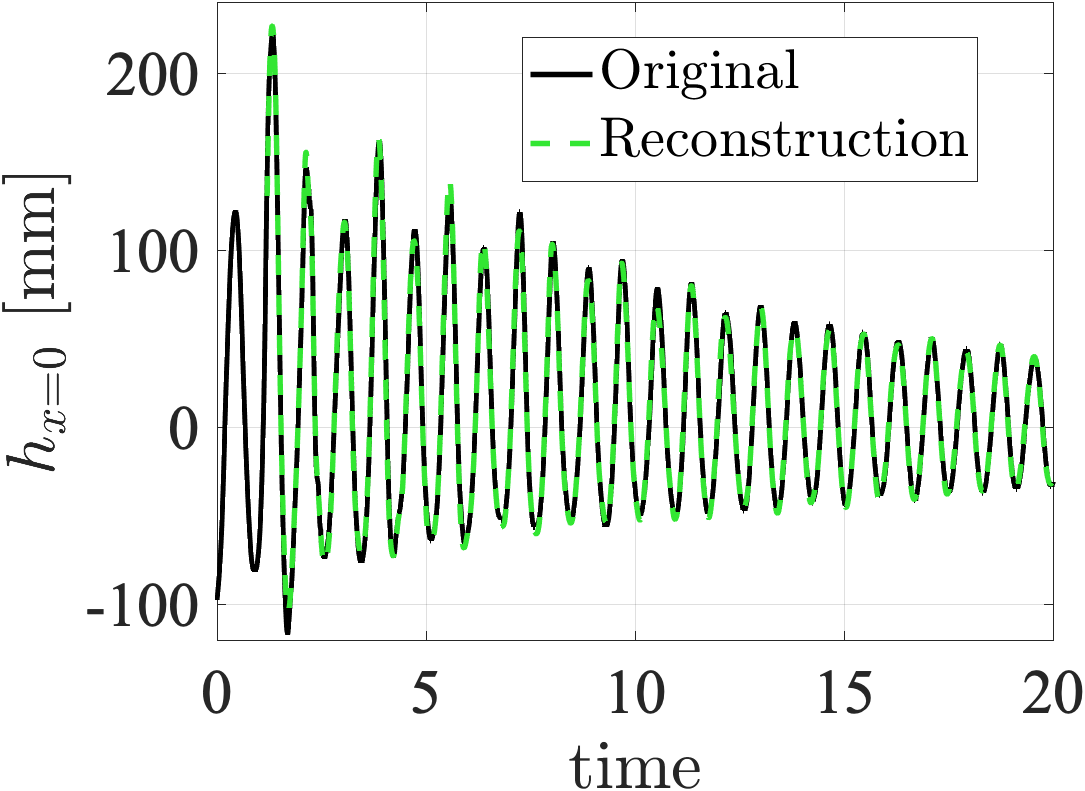}}\hfill
	\subfloat[\label{fig:sloshingdecay3}]{\includegraphics[width=0.33\linewidth]{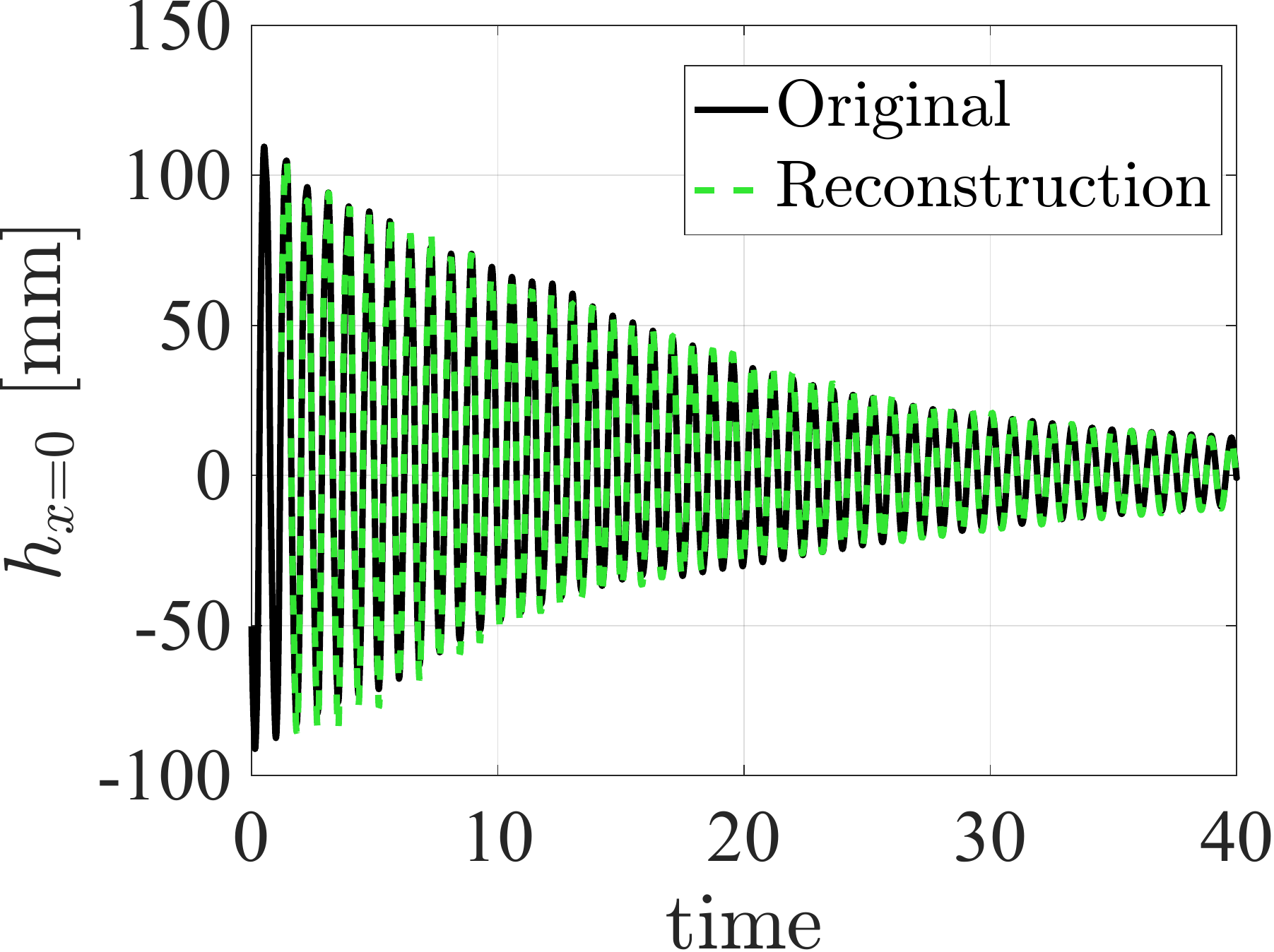}}
	\subfloat[\label{fig:sloshingphaseportrait}]{\includegraphics[width=0.33\linewidth]{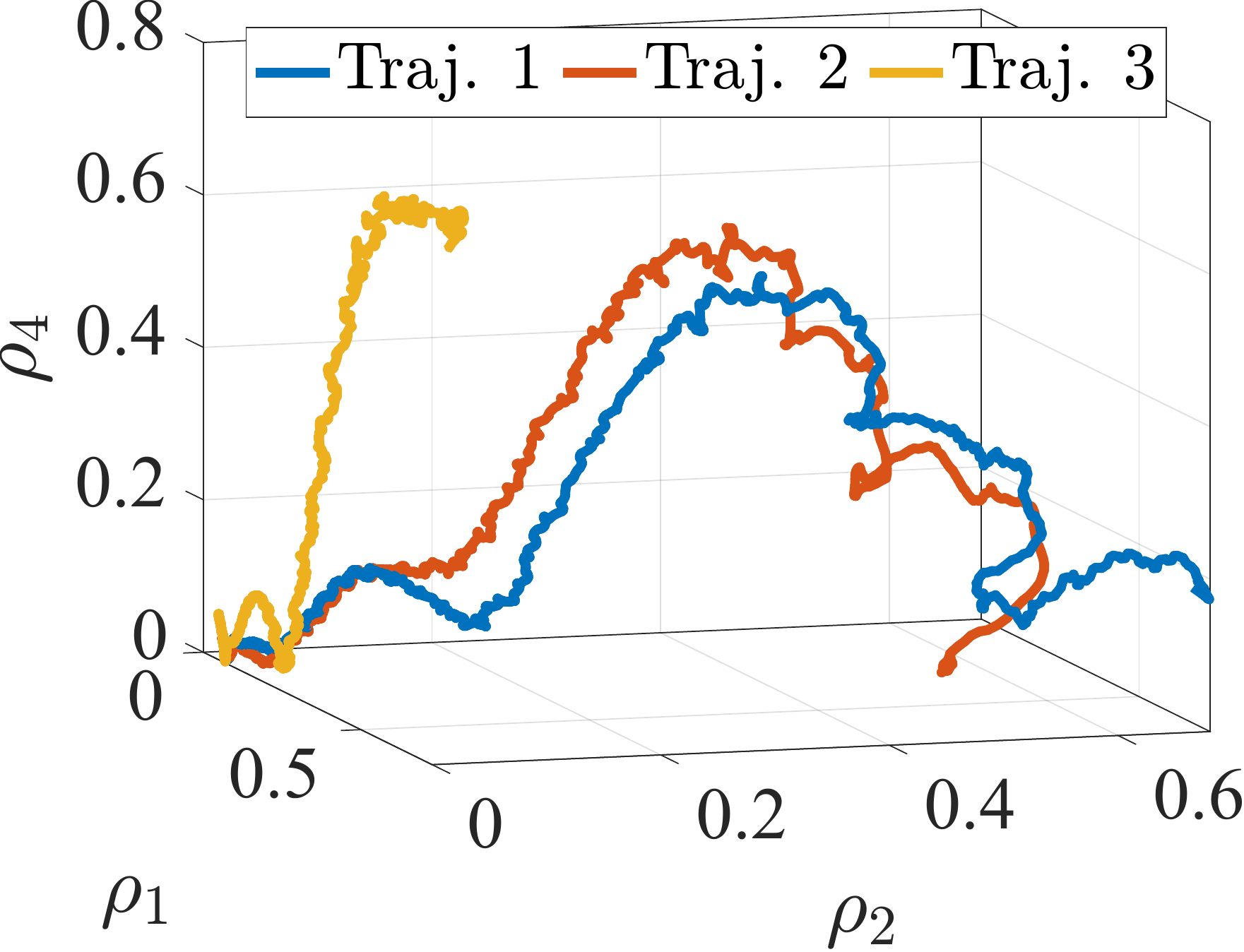}}\hfill
	\caption{The prediction on the 6D SSM for the decay of (a) Trajectory 1 and (b) Trajectory 3. (c) Phase portrait of the amplitudes of the normal form coordinates on the SSM for each of the trajectories shows the modal contributions and development for different intial flow states.}\label{fig:sloshingpred}
\end{figure}

We project the training trajectories onto the SSM and transform them to the normal form in polar coordinates.
The development of the amplitudes in the normal form are shown in Fig.~\ref{fig:sloshingphaseportrait} for each trajectory.
This plot suggests that
(i) the wave-breaking motion (Traj.~3) does not seem to have any significant content of the second mode, 
(ii) the amplitude of the fourth mode indeed increases after motor detachment, 
(iii) there is a small oscillation in these signals not captured by our model, which may be due to noise, insufficient separation of the modal subspaces, a mode outside our model, or some other phenomenon. 

\begin{figure}[h]
	\subfloat{\includegraphics[width=0.36\linewidth]{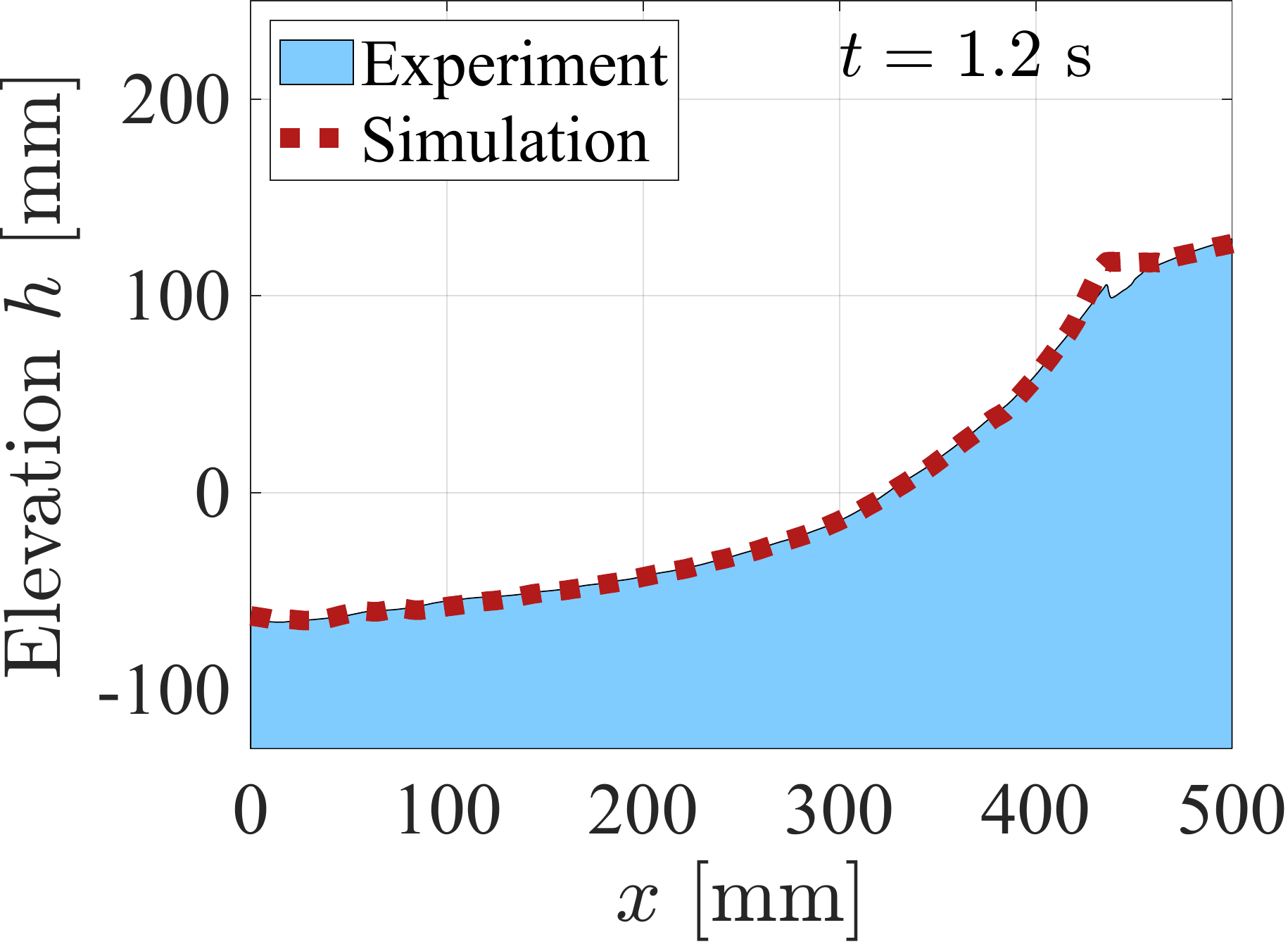}}\hfill
	\subfloat{\includegraphics[width=0.31\linewidth]{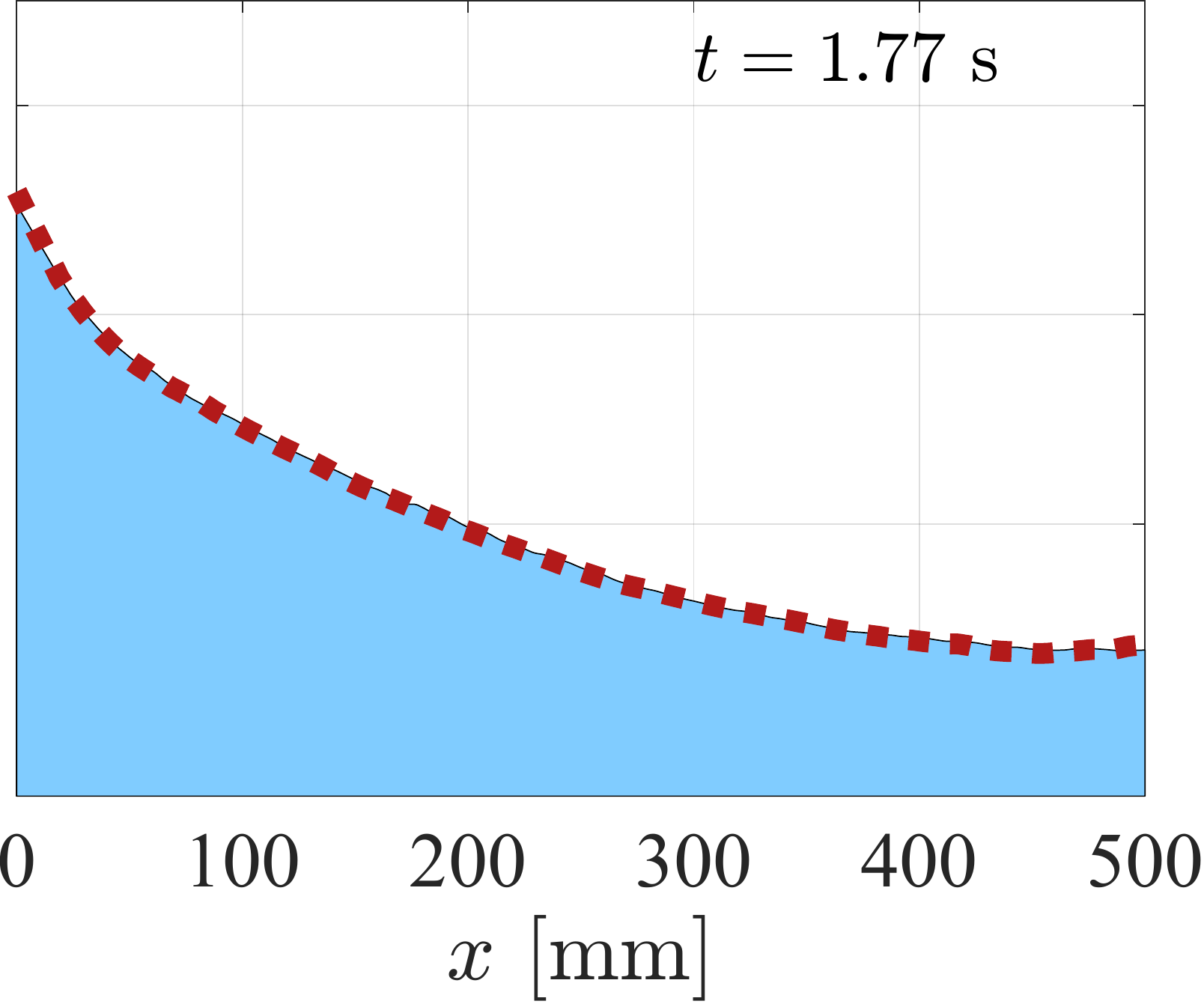}}\hfill
	\subfloat{\includegraphics[width=0.31\linewidth]{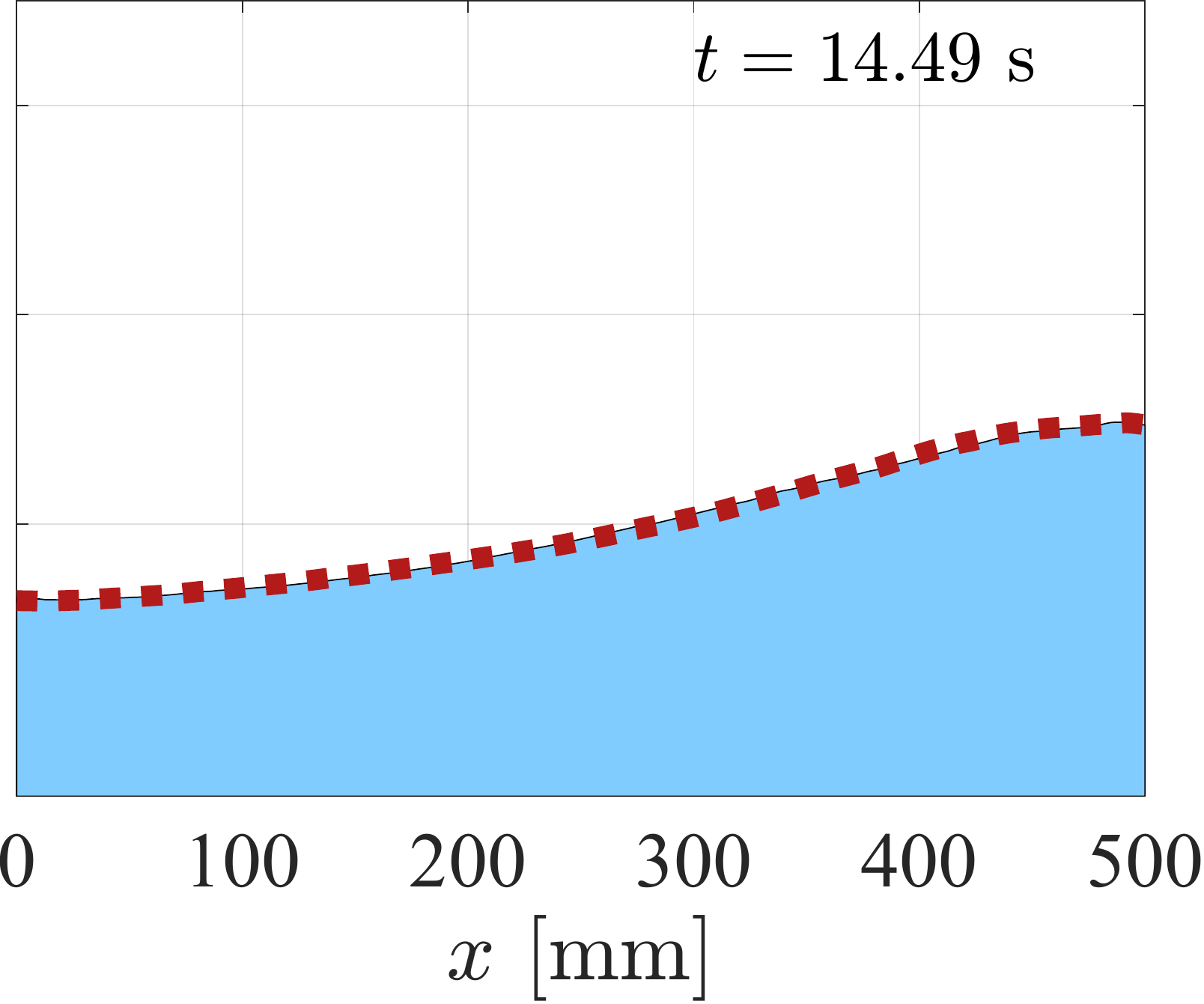}}
	\caption{The experimentally measured surface profile decay agrees with our 6D SSM model prediction for Trajectory 2.}\label{fig:sloshingprofile123}
\end{figure}

We note that the combined higher modal content in the signal is small - only about 10~\% with respect to the first mode.
This is because the data is decaying from steady states induced by forcing near the first eigenfrequency.
Due to their symmetric shape, isolated forcing of the second and fourth modes is not possible with horizontal harmonic excitation.
Nevertheless, we are able to capture these smaller oscillations on the SSM.
The key technology allowing this enhancement is the enforcement of the delay-embedded tangent space in our SSM reconstruction, based on the theoretical eigenfrequencies and mode shapes.

Due to the small activation of the higher modes, our model is expected to be sensitive to noise.
It is, however, stable with respect to changes in starting time, manifold order, and delay parameters.
A more robust model can be obtained by decreasing the manifold dimension to 4, neglecting the relatively minor influence of the fourth mode, resulting in an average NMTE of 3.9~\%.
Here, since our objective was modal analysis of different flow states, we chose the more detailed 6D model.

\section{Conclusions}\label{sec:conclusions}
We have shown that for a scalar observation of an invariant manifold tangent to a spectral subspace at a fixed point, the delay-embedded reconstruction of the tangent space is dependent only on the corresponding eigenvalues of the full system linearized at that point.
In particular, we have proven that the columns of a \van{}, given by repeated multiplication of the exponential of the eigenvalues times the timelag, are eigenvectors for the linearized system in the observable space.
Therefore, the \van{} diagonalizes the linear part of the delay-embedded dynamics. 
We have also shown that when several quantities are measured and delay-embedded simultaneously, the tangent space can be expressed given the \van{} and the mode shapes expressed in the observable function components. 
These results hold for any invariant manifold tangent to a modal subspace with distinct eigenvalues, including, e.g., classic stable manifolds.
Here, our focus was the application of this result to spectral submanifolds of hyperbolic fixed points.

In an attempt to exploit this uncovered structure, we have shown that for data-driven SSM model reduction, when the eigenvalues are approximately known, we can analytically predict the tangent space of the embedded SSM a~priori to achieve local modal decomposition and aid the reconstruction of the nonlinear reduced dynamics. 
We have found that even for small activation of higher modes, this trick helps modeling complex multimodal nonlinear dynamics on an SSM, which in turn allows for analysis of modal energy interchange and instantaneous frequencies.

Our theory assumes a generic observable function, which we describe in more detail in our first and second example, and distinct eigenvalues. 
While the second assumption is a generic one in a mathematical sense, it is not always satisfied for engineering structures with symmetry. 
Vibrations in a square plate is an example where our theory would fail, as it has repeated eigenvalues.
Using a vector-valued observable may help in differentiating between the modes in such a case.
Further, while technically covered by the theory, possible practical difficulties related to the conditioning of the \van{} include highly different or very similar eigenvalues, or eigenvalues of different stability type.

In our third example with data from experiments, we also devised a new heuristic scheme for using delay embedding to study modal contents in a signal. 
With this method, that served as an initial guess, we projected the data onto the respective prescribed modal subspaces, thereby implicitly assuming that the SSM of each mode is nearly flat.
An interesting development of this idea would be its use as a filter, which could remove or keep certain frequencies in a signal.
Another idea would be to use the tangent space condition as a verification or for iterative adjustment of the linear fit of the reduced dynamics.
Finally, in analogy with a Fourier analysis, it would be possible to estimate both instantaneous frequency and damping of a signal by singular value decomposition of the trajectory in delay coordinates followed by analysis of the Vandermonde matrix columns.
This ties in with several other observations made in the literature; for example, for a linear system, these columns agree with the recently proposed notion of principal component trajectories \cite{dylewsky22}.
Overall, we believe that our findings shed more light on delay-embedding invariant manifolds and selecting delay parameters in particular.
For that reason, we expect these results to be of use for a wide range of data-driven methods. 

\subsection*{Acknowledgements}
We are grateful to Kerstin Avila and Bastian Bäuerlein (U. Bremen) for sharing their experimental surface profile data from Ref.~\cite{bauerlein21} with us.

\appendix{}
\section{Appendix}\label{sec:proofs}
\subsection{Proof of Theorem 1}\label{sec:tangentproof}
Let $\mfd$ be a $\ssmdim$-dimensional invariant manifold of \eqref{eq:fullsystem} containing the origin of $\R^\fulldim$.
The tangent space of $\mfd$ at the origin can be written $T_{\vct 0} \mfd = \spann{\{\eigv_\modi\}_{\modi\in \modiset}}$, where $\modiset \subset \{1,\dots,\fulldim\}$ is an index set labeling the $\ssmdim$ eigenvectors $\eigv_\modi$ spanning the spectral subspace from which the manifold is emanating. 
For example, for a stable manifold, $\modiset = \{\modi : \real{\lambda_{\modi}}<0\}$.
We also assume that the eigenvalues in question are distinct, $\geni\ne \modi \Leftrightarrow \lambda_{\geni} \ne \lambda_{\modi}$, for $\geni,\modi \in \modiset$.

To simplify the notation, we transform the full state space to modal coordinates \eqref{eq:modsys}. 
We rewrite the observable function on the system $\vct{\fullc} \in \R^\fulldim$ as an observable on the modal coordinate system $\vct \modc \in \Ce^\fulldim$ as $\hat{\obs}(\vct \modc) = \obs(\Eigv \vct \modc)$. 
We denote by $\ssmflow^t = \Eigv \circ \fullflow^t \circ \Eigv^{-1}$ the flow in $\Ce^\fulldim$. 
Consider the sampling map in modal coordinates 
\begin{equation}
	\hat{\sampmap} = \left[\begin{array}{c}
	\hat{\obs} \\
	\hat{\obs}\circ\ssmflow^\tau \\
	\hat{\obs}\circ\ssmflow^{2\tau} \\
	\vdots \\
	\hat{\obs}\circ\ssmflow^{(\deldim-1)\tau}\end{array}\right]
	 : \Ce^\fulldim \to \R^{\deldim},\quad \vct{\delc} = \hat{\sampmap}(\vct{\modc}).
\end{equation}
Under the conditions of Takens's embedding theorem, the delay embedding map $\vct\Psi = \hat{\sampmap}|_{\mfd}: \mfd \to \mfdo$ is a smooth embedding with a smooth inverse $\vct\Psi^{-1}: \mfdo \to \mfd$, and $\vct\Psi(\vct 0) = \delp$. 

In order to derive the tangent space $T_{\delp}\mfdo$, we compute the derivative of the embedding at $\vct 0$:
\begin{equation}
	\D\vct\Psi(\vct 0) = 
	\left[\begin{array}{c}
	\D\hat{\obs}(\vct 0) \\
	\D\hat{\obs}(\ssmflow^\tau(\vct 0)) \circ \D\ssmflow^\tau(\vct 0) \\
	\vdots \\
	\D\hat{\obs}(\ssmflow^{(\deldim-1)\tau}(\vct 0))\circ \D\ssmflow^{(\deldim-1)\tau}(\vct 0)
	\end{array}\right]
	= 
	\left[\begin{array}{c}
	\D\hat{\obs}(\vct 0) \\
	\D\hat{\obs}(\vct 0) \circ \e^{\vct\Lambda\tau} \\
	\vdots \\
	\D\hat{\obs}(\vct 0)\circ \e^{\vct\Lambda(\deldim-1)\tau}
	\end{array}\right].
\end{equation}
Now note that the $\deli$th component expressed in modal coordinates is
\begin{equation}
	\D\hat{\obs}(\vct 0)\circ \e^{\vct\Lambda \deli\tau} (\vct \modc)
	= \sum_{\modi\in \modiset} \left.\frac{\partial \hat{\obs}}{\partial \modc_{\modi}}\right|_{\vct 0} \e^{\lambda_{\modi} \deli\tau} \modc_{\modi},\quad \deli\in \{1,\dots,\deldim\}.
\end{equation} 
We define the \van{} $\vct{\Van}$ of the $\ssmdim$ eigenvalues $\left\{\lambda_{\modi}\right\}_{\modi \in \modiset}$ governing the linearized dynamics on $\mfd$ as $\Van_{\deli\modi} = \e^{\lambda_{\modi} \deli\tau}$.
We conclude that the tangent space of the observable manifold at the fixed point in modal coordinates can be written as
\begin{equation}
	T_{\delp}\mfdo 
	= \{\D\vct\Psi(\vct 0)\vct \modc,\ \vct \modc\in \Ce^{\fulldim}\} 
	= \range{\left\{
	\vct{\Van}
	\diag\left(\left.\frac{\partial \hat{\obs}}{\partial \vct \modc}\right|_{\vct 0}\right)
	\right\}} 
	= \range{\vct{\Van}},
\end{equation}
where the diagonal matrix acts only as a rescaling of each component of $\vct \modc$.
Therefore, one matrix representation of the tangent space of the manifold in the observable space is $\vct{\Van}$, which is independent both of the matrix $\Eigv$ of full system eigenvectors and the observable function $\obs$. 

Note that we must have $\frac{\partial \hat{\obs}}{\partial \modc_{\modi}}|_{\vct 0} \ne 0\quad \forall \modi\in \modiset$, which defines the genericity of $\obs$. 
In practice, this implies that the linearized observable function must contain contributions from all modal coordinates that we wish to model.
In addition, note that for the embedding of the tangent space itself, i.e., the linear system, $\deldim=\ssmdim$ suffices.

\subsection{Proof of Theorem 2}\label{sec:diagproof}
The flow on $\mfdo$ is
\begin{equation}
	\tilde{\ssmflow}^t = \vct \Psi\circ \ssmflow^t \circ \vct \Psi^{-1}.
\end{equation}
We compute the ODE on $\mfdo$, $\dot{\vct{\delc}} = \delf(\vct{\delc})$, as
\begin{equation}
	\delf = \frac{d}{dt}\tilde{\ssmflow}^t = \D\vct\Psi\circ \modf \circ \vct\Psi^{-1},
\end{equation}
where $\modf$ is given by (\ref{eq:modsys}).
The derivative of $\delf$ at the fixed point is, therefore,
\begin{equation}
\begin{aligned}
	\D \delf(\delp) 
	&= \D\vct\Psi(\vct 0)\circ \vct\Lambda \circ \D\vct\Psi^{-1}(\delp)
	 = \vct{\Van}
	 \diag\left(\left.\frac{\partial \hat{\obs}}{\partial \vct \modc}\right|_{\vct 0}\right) 
	 \vct\Lambda 
	 \diag\left(\left.\frac{\partial \hat{\obs}}{\partial \vct \modc}\right|_{\vct 0}\right)^{-1} 
	 \vct{\Van}^{\dagger} \\
	&= \vct{\Van} \vct\Lambda \vct{\Van}^{\dagger},
\end{aligned}
\end{equation}
where we used the commutative property of multiplication of diagonal matrices, which eliminates the linearized observable function terms, and the fact that $\D\vct \Psi^{-1}(\delp) = \diag\left(\left.\frac{\partial \hat{\obs}}{\partial \vct \modc}\right|_{\vct 0}\right)^{-1} \vct{\Van}^{\dagger}$ is well-defined.
To see this, note that the derivative of the delay embedding map composed with its inverse 
\begin{equation}
	\D\vct \Psi(\vct 0)\circ \D\vct \Psi^{-1}(\delp)
	= \vct{\Van}\diag\left(\left.\frac{\partial \hat{\obs}}{\partial \vct \modc}\right|_{\vct 0}\right)
	\diag\left(\left.\frac{\partial \hat{\obs}}{\partial \vct \modc}\right|_{\vct 0}\right)^{-1} 
	 \vct{\Van}^{\dagger}
	 =
	 \vct{\Van}\vct{\Van}^{\dagger},
\end{equation}
maps all points in $T_{\vct 0}\mfd$ to themselves, since $\vct{\Van}$ has full rank under the assumption of distinct eigenvalues $\left\{\lambda_\modi\right\}_{\modi\in\modiset}$.

Taylor-expanding the ODE on $\mfdo$ in the observable space yields
\begin{equation}
	\dot{\vct{\delc}} = \D \delf(\delp)(\vct{\delc}-\delp) + \lilordo{|\vct{\delc}-\delp|} = \vct{\Van} \vct\Lambda \vct{\Van}^{\dagger}(\vct{\delc}-\delp) + \lilordo{|\vct{\delc}-\delp|}.
\end{equation}
Therefore, under the assumptions of a generic observable function and distinct eigenvalues, the tangent space $T_{\delp}\mfdo$ and the linearized dynamics $\D \delf(\delp)$ in the observable space $\R^\deldim$ are both fully determined by the timelag $\tau$ and the eigenvalues $\lambda_{\modi}$, $\modi \in \modiset$. 

In the special case that $\mfd = \R^\fulldim$, if $\deldim\ge 2\fulldim+1$, the entire phase space can be reconstructed.
For a linear system, $\deldim=\fulldim$ suffices, and the delay embedding reduces to a linear operator.

\subsection{Proof of Theorem 3}\label{sec:multivarproof}
For a vector-valued observable, $\vct \obs : \R^{\fulldim} \to \R^{\obsdim}$, the delay embedding map reads
\begin{equation}
	\vct\Psi = 
	\left[\begin{array}{c}
	\vct\Psi_{\obs_1} \\ \vdots \\ \vct\Psi_{\obs_{\obsdim}}
	\end{array}\right]
	: \mfd \to \mfdo \subset \R^{\deldim\obsdim},\quad
	\D\vct\Psi(\vct 0) = 
	\left[\begin{array}{c}
	\D\vct\Psi_{\obs_1}(\vct 0) \\ \vdots \\ \D\vct\Psi_{\obs_{\obsdim}}(\vct 0)
	\end{array}\right],
\end{equation}
where $\vct\Psi_{\obs_{\obsi}}$ is the delay embedding map corresponding to the $\obsi$th component of the observable function $\vct\obs$.
The derivatives are given by
\begin{equation}
	\D\vct\Psi_{\obs_{\obsi}}(\vct 0) 
	= \vct{\Van}
	\diag\left(\left.\frac{\partial \hat{\obs}_{\obsi}}{\partial \vct \modc}\right|_{\vct 0}\right).
\end{equation}
In this case, the tangent space is not independent of the observable function. 
Instead, it is affected by the relative dependency of each component $\obs_{\obsi}$ of the observable function on each modal coordinate $\modc_{\modi}$.
The tangent space can, however, be expressed as
\begin{equation}
	T_{\delp}\mfdo = 
	\range
	\left[\begin{array}{c}
	\vct{\Van}
	\diag\left(\left.\frac{\partial \hat{\obs}_1}{\partial \vct \modc}\right|_{\vct 0}\right) \\ 
	\vdots \\ 
	\vct{\Van}
	\diag\left(\left.\frac{\partial \hat{\obs}_{\obsdim}}{\partial \vct \modc}\right|_{\vct 0}\right)
	\end{array}\right].
\end{equation}

\printbibliography

\end{document}